\documentclass{article}
\usepackage{amsmath, amssymb, amsthm}
\usepackage{enumerate}
\usepackage{url}
\usepackage{tikz}

\theoremstyle{plain}
\newtheorem{thm}{Theorem}
\newtheorem{innermainthm}{Theorem}
\newenvironment{mainthm}[1]
  {\renewcommand\theinnermainthm{#1}\innermainthm}
  {\endinnermainthm}
  
\newtheorem{lemma}[thm]{Lemma}
\newtheorem{claim}[thm]{Claim}
 
\newtheorem{prop}[thm]{Proposition} 
\newtheorem{asmp}{Assumption}

\theoremstyle{definition}

\theoremstyle{remark}
\newtheorem{rem}[thm]{Remark}

\numberwithin{equation}{section}
\numberwithin{thm}{section}

\newcommand{\s}{\mathbf s}
\newcommand{\tbf}{\mathbf t}
\newcommand{\ubf}{\mathbf u}
\newcommand{\F}{\mathbb F}
\newcommand{\R}{\mathbb R}
\newcommand{\Gw}{G_{\mathrm{walk}}}
\newcommand{\G}{\mathsf G}
\newcommand{\Hsf}{\mathsf H}
\DeclareMathOperator{\Cay}{Cay}

\title{Hypergraph expanders of all uniformities from Cayley graphs}
\author{David Conlon\thanks{Department of Mathematics, California Institute of Technology, Pasadena, CA 91125. Email: {\tt dconlon@caltech.edu}. Research supported by a Royal Society University Research Fellowship and ERC Starting Grant 676632.} \and Jonathan Tidor\thanks{Department of Mathematics, Massachusetts Institute of Technology, Cambridge, MA 02139. Email: {\tt jtidor@mit.edu}. Researched supported by an MIT Presidential Fellowship.} \and Yufei Zhao\thanks{Department of Mathematics, Massachusetts Institute of Technology, Cambridge, MA 02139. Email: {\tt yufeiz@mit.edu}. Research supported by NSF Awards DMS-1764176 and DMS-1362326 and the MIT Solomon Buchsbaum Fund.}}
\date{}

\begin{document}

\maketitle

\begin{abstract}
Hypergraph expanders are hypergraphs with surprising, non-intuitive expansion properties. In a recent paper, the first author gave a simple construction, which can be randomized, of $3$-uniform hypergraph expanders with polylogarithmic degree. We generalize this construction, giving a simple construction of $r$-uniform hypergraph expanders for all $r \geq 3$. 
\end{abstract}

\section{Introduction}
\label{sec:intro}

An expander is a sparse graph with the property that every small vertex subset expands, that is, is adjacent to many vertices outside of the set. The study of expanders has occupied a central place in both mathematics and computer science for the last forty years, finding numerous applications across both areas. We refer the reader to the detailed surveys~\cite{HLW06} and \cite{L12} for further information.

That expanders exist was first shown by Pinsker~\cite{P73}, who observed that random regular graphs are almost surely expanders. Since his result, many different methods have been developed for constructing expanders, some completely explicit and some retaining elements of randomness: using representation theory and Kazhdan's property (T), as done by Margulis~\cite{M73} when he found the first explicit construction; by taking finite quotients of an infinite tree, a procedure used by Margulis~\cite{M88} and by Lubotzky, Phillips and Sarnak~\cite{LPS88} to produce Ramanujan graphs, graphs with optimal spectral properties; the zig-zag product approach of Reingold, Vadhan and Wigderson~\cite{RVW02}; through random lifts~\cite{BL06}, an idea which has recently led to constructions of bipartite Ramanujan graphs~\cite{MSS15} of all degrees; and by taking random Cayley graphs~\cite{AR94, BGGT15}. Since the last of these will be important to us in what follows, let us say a little more.

The basic result on the expansion of random Cayley graphs is due to Alon and Roichman~\cite{AR94}. Their result is best explained in terms of spectral properties. Writing $\mathbf A_G$ for the adjacency matrix of a graph $G$, we let $\lambda(G)$ be the maximum absolute value of a non-trivial eigenvalue of $\mathbf A_G$. We then say that a $d$-regular graph $G$ is an \emph{$\epsilon$-expander} if $\lambda(G)\leq (1-\epsilon)d$. By Cheeger's inequality, a standard result in the area, the spectral condition implies that every vertex subset $U$ of $G$ with $|U| \leq |V(G)|/2$ has at least $\frac{\epsilon}{2} |U|$ neighbours outside of $U$, so $G$ conforms with our intuitive idea of what an expander should be. The Alon--Roichman theorem now says that for any $0 < \epsilon < 1$ there exists $C$ such that if a finite group has $n$ elements then the Cayley graph generated by $C \log n$ random elements of this group is almost surely an $\epsilon$-expander. This theorem is easily seen to be best possible when the group is $\F_2^t$, the case which will be of most relevance to us here.

In recent years, a theory of hypergraph, or high-dimensional, expanders has begun to emerge. The literature in this area is already quite large, but a good place to start exploring might be with the recent ICM survey by Lubotzky~\cite{L18}. Many different definitions have been proposed for hypergraph expansion, each with their own strengths and weaknesses. Intriguingly, one of the few points of consensus among these definitions is that random hypergraphs are not good models for hypergraph expansion.

What then is a good model? The prototype for all subsequent constructions (and, in effect, the foundation on which the entire area is built) are the Ramanujan complexes of Lubotzky, Samuels and Vishne~\cite{LSV05}, Li~\cite{Li04} and Sarveniazi~\cite{S07}. In the same way that Ramanujan graphs are built as finite quotients of infinite trees, inheriting many properties of these trees, Ramanujan complexes are finite quotients of Bruhat--Tits buildings and again inherit properties from these buildings.

To say more about their properties, let us fix some terminology. An \emph{$r$-uniform hypergraph} is a pair $H=(V,E)$ where $E\subseteq\binom{V}{r}$. The elements of $V$ are called vertices and the elements of $E$ are called $r$-edges. A $k$-edge of $H$ is a $k$-element subset of an $r$-edge. A hypergraph is \emph{$D$-regular} if every $(r-1)$-edge is contained in exactly $D$ of the $r$-edges. Then one property of Ramanujan complexes is that they are regular.

This is already a non-trivial property to obtain, especially if the number of edges in the hypergraph is linear in the number of vertices, as it is for Ramanujan complexes. But one might hope for more. The property that has drawn the most attention in the literature is Gromov's notion of topological expansion~\cite{G10} and its weaker relative, geometric expansion. We say that an $r$-uniform hypergraph $H$, seen as a simplicial complex, is an \emph{$\epsilon$-topological expander} if for any continuous map $\varphi$ from the complex into $\R^{r-1}$ there exists a point $p$ such that $\varphi^{-1}(p)$ intersects an $\epsilon$-fraction of the edges of $H$. Geometric expansion is defined similarly but only needs to hold for affine maps $\phi$ defined by first mapping the vertices of $H$ and then extending the map to the convex hull using linearity. That Ramanujan complexes are geometric expanders was shown in~\cite{FGLNP12}, while topological expansion (of a hypergraph derived from Ramanujan complexes) was shown in~\cite{EK17, KKL16}.

The property we will be particularly concerned with here is very different, but also very natural, being a generalization of the notion of graph expansion defined earlier. Given an $r$-uniform hypergraph $H$, define its \emph{walk graph} $\Gw(H)$ to be the graph whose vertices are the $(r-1)$-edges of $H$ with an edge between two $(r-1)$-edges if they are contained in a common $r$-edge. We say that $H$ is an \emph{$\epsilon$-expander} if $\Gw(H)$ is an $\epsilon$-expander. In particular, this implies that the random walk defined by starting at any $(r-1)$-edge and then repeatedly moving to an adjacent $(r-1)$-edge chosen uniformly at random converges rapidly to the uniform distribution. For Ramanujan complexes, this latter property was verified by Kaufman and Mass~\cite{KM16}.

The main result of this paper is a simple construction of such expanders for all uniformities $r \geq 3$, building significantly on earlier work of the first author~\cite{C17} which applied in the $3$-uniform case. The construction in~\cite{C17} is surprisingly simple: given an expanding Cayley graph over $\F_2^t$ with generating set $S$, define $\Hsf$ to be the hypergraph with vertex set $\F_2^t$ and 3-edges $(x+s_1,x+s_2,x+s_3)$ for $x\in\F_2^t$ and $s_1,s_2,s_3\in S$ distinct. Then $\Hsf$ is a hypergraph expander. For higher uniformities, the construction is not quite so simple, though it also not much more complicated.

\subsection{The construction}
\label{ssec:const}

We work with the following parameters: $r\geq 3$, the uniformity of our hypergraph; $t$, the dimension of our base vector space; $S\subset\F_2^t$, the generating set of our underlying Cayley graph; and $\epsilon>0$, the degree of expansion of this Cayley graph. We will typically make the following three assumptions.

\begin{asmp}
\label{asmp1}
For $r\geq 3$, $t$ and $S\subset\F_2^t$, the set $S$ satisfies $|S|\geq 2^{2r}$.
\end{asmp}

\begin{asmp}
\label{asmp2}
For $r\geq 3$, $t$ and $S\subset\F_2^t$, the sum $s_1+\cdots+s_l$ is different for every choice of $0\leq l\leq 2^r$ and every choice of distinct $\{s_1,\ldots,s_l\}\subset S$.
\end{asmp}

\begin{asmp}
\label{asmp3}
For $t$, $\epsilon>0$ and $S\subset\F_2^t$, the graph $\Cay(\F_2^t,S)$ is an $\epsilon$-expander.
\end{asmp}

In practice, for $r \geq 3$ and $\epsilon > 0$ fixed, we will let $C$ be a sufficiently large constant in terms of $r$ and $\epsilon$ and take $t$ sufficiently large in terms of $C$, $r$ and $\epsilon$. If we then define $S$ to be a uniform random set $S\subset\F_2^t$ of size $|S|=Ct$, the assumptions are all satisfied with high probability. Indeed, Assumption 1 is obvious, Assumption 2 follows from a first moment calculation and Assumption 3 follows from the Alon--Roichman theorem~\cite{AR94}.

Let $\mathcal P$ be a collection of subsets of $[r]$ defined as follows:
\begin{equation}
\label{eq:P}
\mathcal P=\begin{cases}
\{I\subset[r]:1\leq |I|<r/2\} &\text{if }r\text{ is odd},\\
\{I\subset[r]:1\leq|I|<r/2\text{ or }|I|=r/2\text{ and }1\in I\} \quad &\text{if }r\text{ is even}.
\end{cases}
\end{equation}
Observe that $\{\emptyset\}\cup\mathcal P$ is a downward-closed family of subsets of $[r]$ that contains exactly one of $I$ and $I^c$ for each $I \subseteq [r]$.

Let $\mathcal T\subset S^{\mathcal P}$ be the set of $|\mathcal P|$-tuples
\begin{equation}
\label{eq:T}
\mathcal T=\{(s_I)_{I\in\mathcal P}: s_I\in S\text{ distinct}\}.
\end{equation}
For $x\in\F_2^t$ and $\s\in\mathcal T$, define $e(x,\s)$ to be the ordered $r$-tuple \[e(x,\s)=\left(x+\sum_{I\in\mathcal P\atop{i\in I}}s_I\right)_{1\le i \le r}.\] 
In the above equation and throughout this paper we always use the convention that the bold-face letters $\s,\tbf,\ubf$ refer to elements of $\mathcal T$. The coordinates of these tuples are then named by subscripted normal-print letters, e.g., $\s=(s_I)_{I\in\mathcal P}$.

For $r\geq 3$, $t$ and $S\subset \F_2^t$ satisfying Assumptions \ref{asmp1}, \ref{asmp2}, define $\Hsf_{r,t,S}$ to be the $r$-uniform hypergraph whose vertex set is $\F_2^t$ and with $\{v_1,\ldots,v_r\}$ an $r$-edge if and only if there exists some $x\in\F_2^t$ and $\s\in\mathcal T$ such that $e(x,\s)=(v_1,\ldots,v_r)$. Note that Assumption \ref{asmp2} implies that every $r$-tuple $e(x,\s)$ for $\s\in\mathcal T$ has distinct coordinates.

The 3-uniform hypergraph $\Hsf_{3,t,S}$ has 3-edges $\{x+s_1,x+s_2,x+s_3\}$ for $x\in\F_2^t$ and $s_1,s_2,s_3\in S$ distinct. This is exactly the hypergraph considered in~\cite{C17}. The $r=4$ case will be spelled out explicitly for the sake of illustration in the next subsection.

Recall that for an $r$-uniform hypergraph $H$, we use $\Gw(H)$ to denote the graph whose vertices are the $(r-1)$-edges of $H$ where two $(r-1$)-edges are connected if and only if they are contained in a common $r$-edge of $H$. The main result of this paper is that if $\Cay(\F_2^t,S)$ is an ordinary expander graph then $\Hsf_{r,t,S}$ is a hypergraph expander in the sense that $\Gw(\Hsf_{r,t,S})$ is an expander graph.

\begin{thm} \label{thm:main}
For $r\geq 3$, $t$, $\epsilon>0$ and $S\subset\F_2^t$ satisfying Assumptions \ref{asmp1}, \ref{asmp2}, \ref{asmp3}, there exists a constant $c=c(r)>0$ such that the walk graph $\Gw(\Hsf_{r,t,S})$ is a $c\epsilon$-expander.
\end{thm}

We also prove a generalization of this theorem to all lower-order random walks. For an $r$-uniform hypergraph $H$ and for $1\leq k\leq r-1$, define the $k$-th order walk graph $\Gw^{(k)}(H)$ to be the graph whose vertices are the $k$-edges of $H$ where two $k$-edges are connected if and only if they are contained in a common $(k+1)$-edge of $H$.

\begin{thm}\label{thm:main-lower}
For $r\geq 3$, $t$, $\epsilon>0$ and $S\subset\F_2^t$ satisfying Assumptions \ref{asmp1}, \ref{asmp2}, \ref{asmp3}, there exists a constant $c=c(r)>0$ such that for every $1\leq k\leq r-1$, the $k$-th order walk graph $\Gw^{(k)}(\Hsf_{r,t,S})$ is a $c\epsilon$-expander.
\end{thm}

We also prove a discrepancy result about $\Hsf_{r,t,S}$ similar to the conclusion of a high-dimensional expander mixing lemma. For an $r$-uniform hypergraph $H=(V,E)$ and $V_1,\ldots, V_r\subseteq V$, define $e_H(V_1,\ldots,V_r)$ to be the number of $r$-tuples $(v_1,\ldots,v_r)\in V_1\times\cdots\times V_r$ such that $\{v_1,\ldots,v_r\}\in E$. 

\begin{thm}\label{thm:avg-discrepancy}
For $r\geq 3$, $t$, $\epsilon>0$ and $S\subset \F_2^t$ satisfying Assumptions \ref{asmp1}, \ref{asmp2}, \ref{asmp3} and $V_1,\ldots,V_r\subseteq V$,
\[\begin{split}\left|e_{\Hsf_{r,t,S}}(V_1,\ldots,V_r)\vphantom{\frac{|S|}{(2^t)^{r-1}}}-\right.
&\left.\frac{|S|(|S|-1)\cdots(|S|-(2^{r-1}-2))}{(2^t)^{r-1}}|V_1|\cdots|V_r|\right|\\
&\leq(1-\epsilon)2r|S|^{2^{r-1}-1}(|V_1|\cdots|V_r|)^{1/r}.
\end{split}\]
\end{thm}

When $\epsilon$ is sufficiently close to $1$ in terms of $r$, an argument in \cite{P17, PRT16} allows us to deduce that $\Hsf_{r,t,S}$ is a geometric expander. The idea is simple. By a result of Pach~\cite{P98}, for any set of $n$ points in $\R^{r-1}$, there are $r$ sets $A_1, \dots, A_r$, each of order at least $cn$, and a point $p \in \R^{r-1}$ such that $p$ is contained in the convex hull of $\{a_1, \dots, a_r\}$ for all $a_1, \dots, a_r$ with $a_i \in A_i$ for $i = 1, \dots, r$. But Theorem~\ref{thm:avg-discrepancy} implies that a positive fraction of the edges in $\Hsf_{r,t,S}$ have one vertex in each of $A_1, \dots, A_r$. Since all of these contain $p$, the desired conclusion follows.

\subsection{Proof techniques (the $r=4$ case)}
\label{ssec:examp}

We illustrate the main features of the construction by exploring the $r=4$ case. Fix $S\subset\F_2^t$ such that the sum $s_1+\cdots+s_\ell$ is different for every choice of $0\leq\ell\leq 16$ and $s_1,\ldots,s_\ell\in S$ distinct. Define $$\mathcal T=\{(s_1,s_2,s_3,s_4,s_{12},s_{13},s_{14})\in S^7:s_1,s_2,s_3,s_4,s_{12},s_{13},s_{14}\text{ distinct}\}.$$ 
Then for $x\in\F_2^t$ and $\s\in\mathcal T$, define the ordered 4-tuple
\[e(x,\s)=\begin{array}{llllllllll}
( & x & +s_1 & & & & +s_{12} & +s_{13} & +s_{14} &,\\
& x & & +s_2 & & & +s_{12} & & &,\\
& x & & & +s_3 & & & +s_{13} & &,\\
& x & & & & +s_4 & & & +s_{14} &).
\end{array}\]
The 4-uniform hypergraph $\Hsf_{4,t,S}$ has vertices $\F_2^t$ and 4-edges $\{v_1,v_2,v_3,v_4\}$ where $e(x,\s)=(v_1,v_2,v_3,v_4)$ for some $x\in\F_2^t$ and $\s\in\mathcal T$. 

At first sight it might seem that there is a distinguished vertex in each 4-edge, but this is an artifact of our presentation. For example, consider the ordered 4-tuple
\begin{equation}\label{eq:symmetric-vertices}
\begin{array}{llllllllll}
( & x & & +s_2 & & & +s_{12} & & &,\\
& x & +s_1 & & & & +s_{12} & +s_{13} & +s_{14} &,\\
& x & & & +s_3 & & & +s_{13} & &,\\
& x & & & & +s_4 & & & +s_{14} &).
\end{array}
\end{equation}
If we define $\tbf$ by $t_1=s_2$, $t_2=s_1$, $t_3=s_3$, $t_4=s_4$, $t_{12}=s_{12}$, $t_{13}=s_{14}$, $t_{14}=s_{13}$, and set $y=x+s_{13}+s_{14}$, then one can easily check that
\begin{equation*}
e(y,\tbf)=
\begin{array}{llllllllll}
( & x+s_{13}+s_{14} & +s_2 & & & & +s_{12} & +s_{14} & +s_{13} &,\\
& x+s_{13}+s_{14} & & +s_1 & & & +s_{12} & & &,\\
& x+s_{13}+s_{14} & & & +s_3 & & & +s_{14} & &,\\
& x+s_{13}+s_{14} & & & & +s_4 & & & +s_{13} &)
\end{array}
\end{equation*}
is exactly the same as (\ref{eq:symmetric-vertices}).

Define the graph $\G_{4,t,S}$ to have vertex set $\F_2^t\times\mathcal T$ with an edge between $(x,\s)$ and $(y,\tbf)$ if $e(x,\s)$ and $e(y,\tbf)$ agree in all but one coordinate. This graph is dual to $\Gw(\Hsf_{4,t,S})$ in a sense that will be made precise later on. A theme that recurs throughout this paper is that we will typically prove that $\G_{4,t,S}$ has some desired property and then deduce that the same property holds for $\Gw(\Hsf_{4,t,S})$.

In this case, we wish to prove that $\Gw(\Hsf_{4,t,S})$ is an expander, so we might want to first prove that $\G_{4,t,S}$ is an expander. Let us start with the easier problem of showing that $\G_{4,t,S}$ is connected.

Start at the vertex $(x,\s)=(x,(s_1,s_2,s_3,s_4,s_{12},s_{13},s_{14}))$. This vertex is adjacent to $(x,\s^{(1)})=(x,(t_{12},s_2,s_3,s_4,s_{12},s_{13},s_{14}))$ since
\begin{equation}\label{eq:move-0}
\begin{split}
e(x,\s)&=\begin{array}{llllllllll}
( & x & +s_1 & & & & +s_{12} & +s_{13} & +s_{14}&,\\
& x & & +s_2 & & & +s_{12} & & &,\\
& x & & & +s_3 & & & +s_{13} & &,\\
& x & & & & +s_4 & & & +s_{14} &),
\end{array}
\\
\\
e(x,\s^{(1)})&=\begin{array}{llllllllll}
( & x & +t_{12} & & & & +s_{12} & +s_{13} & +s_{14}&,\\
& x & & +s_2 & & & +s_{12} & & &,\\
& x & & & +s_3 & & & +s_{13} & &,\\
& x & & & & +s_4 & & & +s_{14} &)
\end{array}
\end{split}
\end{equation}
agree in all but the first coordinate. 

Vertex $(x,\s^{(1)})$ is adjacent to $(x,\s^{(2)})=(x,(s_{12},s_2,s_3,s_4,t_{12},s_{13},s_{14}))$ since
\begin{equation}\label{eq:move-1}
\begin{split}
e(x,\s^{(1)})&=\begin{array}{llllllllll}
( & x & +t_{12} & & & & +s_{12} & +s_{13} & +s_{14}&,\\
& x & & +s_2 & & & +s_{12} & & &,\\
& x & & & +s_3 & & & +s_{13} & &,\\
& x & & & & +s_4 & & & +s_{14} &),
\end{array}
\\
\\
e(x,\s^{(2)})&=\begin{array}{llllllllll}
( & x & +s_{12} & & & & +t_{12} & +s_{13} & +s_{14}&,\\
& x & & +s_2 & & & +t_{12} & & &,\\
& x & & & +s_3 & & & +s_{13} & &,\\
& x & & & & +s_4 & & & +s_{14} &)
\end{array}
\end{split}
\end{equation}
agree in all but the second coordinate.

Repeating these two steps two more times produces a path of length four from $(x,\s^{(2)})$ to $(x,\s^{(6)})=(x,(s_{14},s_2,s_3,s_4,t_{12},t_{13},t_{14}))$, where $t_{12},t_{13},t_{14}$ are arbitrary elements of $S$. Repeating a step similar to (\ref{eq:move-0}) four times produces a path of length four from $(x,\s^{(6)})$ to $(x,\s^{(10)})=(x,(t_1,t_2,t_3,t_4,t_{12},t_{13},t_{14}))$, where $t_1, t_2,\ldots,t_{14}$ are arbitrary.

We call this procedure `Boolean bubbling' because it has the following interpretation. The coordinates of $\s$ are indexed by the bottom half of the Boolean lattice of rank four. Now the above procedure allows us to insert a new value at the bottom of the Boolean lattice (\ref{eq:move-0}) and then `bubble' it up the lattice (\ref{eq:move-1}). Repeating this procedure allows us to change the value of $\s$ arbitrarily.\footnote{Note that (\ref{eq:move-0}) and (\ref{eq:move-1}) are essentially the same operation. The latter switches two coordinates, $s_I$ and $s_{I\sqcup\{j\}}$, while the former can be thought of as switching $s_{\emptyset}$ and $s_j$. Since $s_{\emptyset}$ does not affect the value of $e(x,\s)$, its value can be changed arbitrarily. We take advantage of this to insert new values.}

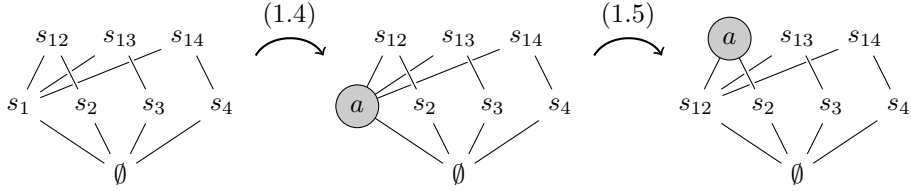
\begin{center}
\begin{tikzpicture}[scale=0.45]
  \node (Aa) at (-2-10,2) {$s_{12}$};
  \node (Ab) at (-10,2) {$s_{13}$};
  \node (Ac) at (2-10,2) {$s_{14}$};
  \node (Ad) at (-3-10,0) {$s_1$};
  \node (Ae) at (-1-10,0) {$s_2$};
  \node (Af) at (1-10,0) {$s_3$};
  \node (Ag) at (3-10,0) {$s_4$};
  \node (Ah) at (0-10,-2) {$\emptyset$};
  \draw (Aa) -- (Ad) -- (Ah) -- (Ae)
  (Ab) -- (Ad) -- (Ac) -- (Ag) -- (Ah) -- (Af);
  \draw[preaction={draw=white, -,line width=2pt}] (Aa) -- (Ae)
  (Ab) -- (Af);

  \node (Ba) at (-2,2) {$s_{12}$};
  \node (Bb) at (0,2) {$s_{13}$};
  \node (Bc) at (2,2) {$s_{14}$};
  \node (Bd) at (-3,0) {$t$};
  \node (Be) at (-1,0) {$s_2$};
  \node (Bf) at (1,0) {$s_3$};
  \node (Bg) at (3,0) {$s_4$};
  \node (Bh) at (0,-2) {$\emptyset$};
  \draw (Ba) -- (Bd) -- (Bh) -- (Be)
  (Bb) -- (Bd) -- (Bc) -- (Bg) -- (Bh) -- (Bf);
  \draw[preaction={draw=white, -,line width=2pt}] (Ba) -- (Be)
  (Bb) -- (Bf);
  \filldraw[fill=gray!40] (Bd) circle (18pt);
  \node at (Bd) {$a$};
  
  \node (Ca) at (-2+10,2) {$t$};
  \node (Cb) at (10,2) {$s_{13}$};
  \node (Cc) at (2+10,2) {$s_{14}$};
  \node (Cd) at (-3+10,0) {$s_{12}$};
  \node (Ce) at (-1+10,0) {$s_2$};
  \node (Cf) at (1+10,0) {$s_3$};
  \node (Cg) at (3+10,0) {$s_4$};
  \node (Ch) at (0+10,-2) {$\emptyset$};
  \draw (Ca) -- (Cd) -- (Ch) -- (Ce)
  (Cb) -- (Cd) -- (Cc) -- (Cg) -- (Ch) -- (Cf);
  \draw[preaction={draw=white, -,line width=2pt}] (Ca) -- (Ce)
  (Cb) -- (Cf);
  \filldraw[fill=gray!40] (Ca) circle (18pt);
  \node at (Ca) {$a$};
  
  \path [->] (-6,1.5) edge [line width=0.8, bend left=60] (-4,1.5);
  \path [->] (4,1.5) edge [line width=0.8, bend left=60] (6,1.5);
  \node at (-5,2.7) {$(\ref{eq:move-0})$};
  \node at (5,2.7) {$(\ref{eq:move-1})$};
\end{tikzpicture}
\end{center}

Let us pause to talk about two details of this procedure that we have not mentioned yet. First, we have to make sure that each new value we `bubble' up the lattice does not disturb previous values we have inserted. This is an easy problem to deal with: as long as we bubble up values to the top level before dealing with the lower level, future steps will not disturb previous ones. Second, recall that the coordinates of $\s$ must be distinct at every step of the process. This might be a problem if we wish to find a path from $(x,\s)$ to $(x,\tbf)$ where, for example, $s_4=t_{12}$. The easiest way to get around this problem is to pick a new $\ubf$ whose coordinates are disjoint from both $\s$ and $\tbf$ and then use the Boolean bubbling procedure to construct a path from $(x,\s)$ to $(x,\ubf)$ and then to $(x,\tbf)$.

Boolean bubbling allows us to construct a path in $\G_{4,t,S}$ from a vertex $(x,\s)$ to any vertex of the form $(x,\tbf)$. To walk to an arbitrary $(y,\tbf)$, we only need one further ingredient.

Given a vertex $(x,\tbf)=(x,(t_1,t_2,t_3,t_4,t_{12},t_{13},t_{14}))$, this vertex is adjacent to $(x',\tbf')=(x+t_3+t_{14},(t_1,t_2,t_{14},t_4,t_{12},t_{13},t_3))$ since 
\begin{equation}\label{eq:move-2}
\begin{split}
e(x,\tbf)&=\begin{array}{llllllllll}
( & x & +t_1 & & & & +t_{12} & +t_{13} & +t_{14}&,\\
& x & & +t_{2} & & & +t_{12} & & &,\\
& x & & & +t_3 & & & +t_{13} & &,\\
& x & & & & +t_4 & & & +t_{14} &),
\end{array}
\\
\\
e(x',\tbf')&=\begin{array}{llllllllll}
( & x+t_3+t_{14} & +t_1 & & & & +t_{12} & +t_{13} & +t_3&,\\
& x+t_3+t_{14} & & +t_2 & & & +t_{12} & & &,\\
& x+t_3+t_{14} & & & +t_{14} & & & +t_{13} & &,\\
& x+t_3+t_{14} & & & & +t_4 & & & +t_3 &)
\end{array}
\end{split}
\end{equation}
agree in all but the second coordinate.

Now we claim that for any $x,y$ which are adjacent in $\Cay(\F_2^t,(S+S)\setminus\{0\})$, there is a path from $(x,\s)$ to $(y,\tbf)$ in $\G_{4,t,S}$. This follows using Boolean bubbling and (\ref{eq:move-2}). Write $y=x+a_1+a_2$ for $a_1, a_2$ distinct elements of $S$. By Boolean bubbling, there is a path from $(x,\s)$ to some $(x,\ubf)$ where $u_3=a_1$ and $u_{14}=a_2$. Then $(x,\ubf)$ is adjacent to $(y,\ubf')$ for some $\ubf'$ by (\ref{eq:move-2}). Finally, another application of Boolean bubbling gives a path from $(y,\ubf')$ to $(y,\tbf)$. Since Assumption~\ref{asmp3} easily implies that $\Cay(\F_2^t,(S+S)\setminus\{0\})$ is also an expander, and in particular connected, this argument shows that $\G_{4,t,S}$ is connected as well.

Here is a convenient rephrasing of this argument which shows how we may extend it to prove that $\G_{4,t,S}$ is an expander. Define $G_{\Cay}$ to be the graph with vertex set $\F_2^t\times\mathcal T$ with an edge between $(x,\s)$ and $(y,\tbf)$ if $x,y$ are adjacent in $\Cay(\F_2^t,(S+S)\setminus\{0\})$. We write $\G_{4,t,S}^M$ for the multigraph which has one edge between $(x,\s)$ and $(y,\tbf)$ for each walk of length $M$ between these vertices in $\G_{4,t,S}$. The argument above shows that for some $M$ there is a copy of $G_{\Cay}$ contained in $\G_{4,t,S}^M$. Since $G_{\Cay}$ is connected, this proves that $\G_{4,t,S}$ is connected.

To prove that $\G_{4,t,S}$ is an expander we need a slightly stronger fact. Write $c\cdot G_{\Cay}$ for the multigraph that has $c$ edges for each edge of $G_{\Cay}$. We need to prove that for some $M$ and $c$ the graphs $\G_{4,t,S}^M$ and $c\cdot G_{\Cay}$ are approximately the same in some appropriate sense. Establishing this fact requires showing that for each $(x,\s)$ and $(y,\tbf)$ adjacent in $G_{\Cay}$ there are many paths between them in $\G_{4,t,S}$. This can be done with a small modification of (\ref{eq:move-1}) and (\ref{eq:move-2}).

Suppose we are in the process of Boolean bubbling, having inserted $t_{12}$ at the bottom of the Boolean lattice. Let $(x,\s)=(x,(t_{12},s_2,s_3,s_4,s_{12},s_{13},s_{14}))$ be the current state. We would normally proceed as in $(\ref{eq:move-1})$ to bubble $t_{12}$ up the lattice. Instead, what we do is simultaneously bubble $t_{12}$ up and insert an arbitrary element $a$ in place of $s_2$. Explicitly, if we set $(x,\s')=(x,(s_{12},a,s_3,s_4,t_{12},s_{13},s_{14}))$, then $(x, \s)$ and $(x, \s')$ are adjacent since
\begin{equation}\label{eq:move-1'}\tag{$\ref{eq:move-1}^\prime$}
\begin{split}
e(x,\s)&=\begin{array}{llllllllll}
( & x & +t_{12} & & & & +s_{12} & +s_{13} & +s_{14}&,\\
& x & & +s_2 & & & +s_{12} & & &,\\
& x & & & +s_3 & & & +s_{13} & &,\\
& x & & & & +s_4 & & & +s_{14} &),
\end{array}
\\
\\
e(x,\s')&=\begin{array}{llllllllll}
( & x & +s_{12} & & & & +t_{12} & +s_{13} & +s_{14}&,\\
& x & & +a & & & +t_{12} & & &,\\
& x & & & +s_3 & & & +s_{13} & &,\\
& x & & & & +s_4 & & & +s_{14} &)
\end{array}
\end{split}
\end{equation}
still agree in all but the second coordinate. Replacing every move of type (\ref{eq:move-1}) with one of type (\ref{eq:move-1'}) in the Boolean bubbling procedure gives not one path of length $M$ from $(x,\s)$ to $(x,\tbf)$, but instead on the order of $|S|^{M-7}$ such paths. Here the exponent $M - 7$ comes from the fact that our modified Boolean bubbling procedure gives on the order of $|S|^M$ paths of length $M$ starting at $(x,\s)$ and ending at a vertex of the form $(x,\tbf)$, while the number of choices for $\tbf$ is on the order of $|S|^7$. The full details of this procedure in the general $r$-uniform construction are given in Section \ref{ssec:bubble}.

Similarly, we upgrade the moves of type $(\ref{eq:move-2})$ by inserting an arbitrary element $a$ at the bottom of the lattice. In particular, suppose we are given a vertex $(x,\tbf)=(x,(t_1,t_2,t_3,t_4,t_{12},t_{13},t_{14}))$. This vertex is adjacent to $(x',\tbf')=(x+t_3+t_{14},(t_1,a,t_{14},t_4,t_{12},t_{13},t_3))$ since 
\begin{equation}\label{eq:move-2'}\tag{$\ref{eq:move-2}^\prime$}
\begin{split}
e(x,\tbf)&=\begin{array}{llllllllll}
( & x & +t_1 & & & & +t_{12} & +t_{13} & +t_{14}&,\\
& x & & +t_{2} & & & +t_{12} & & &,\\
& x & & & +t_3 & & & +t_{13} & &,\\
& x & & & & +t_4 & & & +t_{14} &),
\end{array}
\\
\\
e(x',\tbf')&=\begin{array}{llllllllll}
( & x+t_3+t_{14} & +t_1 & & & & +t_{12} & +t_{13} & +t_3&,\\
& x+t_3+t_{14} & & +a & & & +t_{12} & & &,\\
& x+t_3+t_{14} & & & +t_{14} & & & +t_{13} & &,\\
& x+t_3+t_{14} & & & & +t_4 & & & +t_3 &)
\end{array}
\end{split}
\end{equation}
still agree in all but the second coordinate.
This is the only additional ingredient that we need to deduce that $\G_{4,t,S}$ is an expander from the fact that $G_{\Cay}$ is an expander. 

To complete the proof, all that remains is to deduce the corresponding result for $\Gw(\Hsf_{4,t,S})$. This will be immediate once we write down exactly in what sense $\G_{4,t,S}$ and $\Gw(\Hsf_{4,t,S})$ are dual to each other. We refer the reader to Section \ref{ssec:spectral} for the details.

\section{Proof of main theorem}
\label{sec:proof}

\subsection{Degree properties}
\label{ssec:basic}

We start by proving a couple of basic properties of our hypergraph $\Hsf_{r,t,S}$. We find the number of representations each $r$-edge has of the form $e(x,\s)$ and compute the number of $r$-edges of $\Hsf_{r,t,S}$ containing a given $(r-1)$-edge.

\begin{lemma} \label{thm:distinct-tuple}
Recall $\mathcal T$ from $(\ref{eq:T})$. For $r,t,S$ satisfying Assumption \ref{asmp2}, the ordered $r$-tuples $e(x,\s)$ for $x\in \F_2^t$ and $\s\in \mathcal T$ are all distinct.
\end{lemma}

\begin{proof}
Suppose $e(x,\s)=e(x',\s')$ for $x,x'\in\F_2^t$ and $\s,\s'\in\mathcal T$. Then this implies that
\[e_i(x,\s)+e_j(x,\s)=e_i(x',\s')+e_j(x',\s')\]
for each $1\leq i<j\leq r$. We rewrite this equation as
\[\sum_{I\in\mathcal P\atop{|\{i,j\}\cap I|=1}}s_I=\sum_{I\in\mathcal P\atop{|\{i,j\}\cap I|=1}}s'_I.\] 
By Assumption \ref{asmp2} on $S$, this implies that
\begin{equation}
\label{eq:dt}\{s_I:I\in\mathcal P\text{ with }|\{i,j\}\cap I|=1\}=\{s'_I:I\in\mathcal P\text{ with }|\{i,j\}\cap I|=1\}
\end{equation} 
for each $1\leq i<j\leq r$.\

We claim that the above equation implies that $\s=\s'$. To see this, fix $I\subseteq[r]$ and consider the pairs $(i,j)\in[r]^2$ such that $i\neq j$ and $|\{i,j\}\cap I|=1$. This condition is satisfied if and only if $(i,j)\in (I\times I^c)\cup (I^c\times I)$. Note that the sets $(I\times I^c)\cup (I^c\times I)$ are distinct for all $I\in\mathcal P$ and are non-empty. This is true simply because $\emptyset,[r]\not\in\mathcal P$ and if $I\in\mathcal P$ then one has $I^c\not\in\mathcal P$.

Thus for $I\in\mathcal P$ the element $s_I$ is present on the left-hand side of $(\ref{eq:dt})$ if and only if $(i,j)\in (I\times I^c)\cup (I^c\times I)$ and no other $s_J$ has the same property. Furthermore, $s'_I$ is present on the right-hand side of $(\ref{eq:dt})$ if and only if $(i,j)\in (I\times I^c)\cup (I^c\times I)$ and no other $s'_J$ has the same property. Since the left- and right-hand sides of $(\ref{eq:dt})$ are equal for all $1\leq i<j\leq r$, this implies that $s_I=s'_I$ for all $I\in\mathcal P$.

To complete the proof, note that 
\[x=e_1(x,\s)+\sum_{I\in P\atop{1\in I}}s_I=e_1(x',\s')+\sum_{I\in P\atop{1\in I}}s'_I=x',\]
as required.
\end{proof}

This implies that given an $r$-edge $e=\{v_1,\ldots,v_r\}$ of $\Hsf_{r,t,S}$ there are at most $r!$ possible $e(x,\s)$ that it corresponds to, one for each ordering of its vertices. We claim that this is an equality. This is immediate for $r$ odd. Suppose $e(x,\s)=(v_1,\ldots,v_r)$. For a permutation $\pi\in\mathfrak S_r$ define $\s^{\pi}$ by $s^{\pi}_I=s_{\pi(I)}$ where $\pi(I)=\{\pi(i):i\in I\}$. Then $e(x,\s^{\pi})=(v_{\pi(1)},\ldots,v_{\pi(r)})$. This argument does not work for $r$ even since for some $I\in\mathcal P$, there are $\pi$ such that $\pi(I)\not\in\mathcal P$. However, the same result follows easily from the next lemma.

\begin{lemma}
\label{thm:symmetric-edges}
Fix $r$ even and $S\subset\F_2^t$. Recall $\mathcal T$ from (\ref{eq:T}). For each $(x,\s)\in\F_2^t\times\mathcal T$, there exists $(y,\tbf)\in\F_2^t\times\mathcal T$ such that $e(y,\tbf)$ is the same as $e(x,\s)$ but with the first two coordinates swapped.
\end{lemma}

\begin{proof}
Let $\pi$ be the transposition that swaps 1 and 2. Write $\mathcal P=\mathcal P_1\sqcup\mathcal P_2\sqcup\mathcal P_3$, where $\mathcal P_1=\{I\subset[r]:0<|I|<r/2\}$, $\mathcal P_2=\{I\subset[r]:|I|=r/2\text{ and }1,2\in I\}$ and $\mathcal P_3=\{I\subset[r]:|I|=r/2, 1\in I\text{ and }2\not\in I\}$. Note that for $I\in\mathcal P_1$ one has $\pi(I)\in\mathcal P_1$, while for $I\in\mathcal P_2$ one has $\pi(I)=I\in\mathcal P_2$. Finally, note that for $I\in\mathcal P_3$ one has $\pi(I^c)\in\mathcal P_3$.

Now let
\[t_I=
\begin{cases}s_{\pi(I)}\qquad&\text{if }I\in\mathcal P_1\cup\mathcal P_2,\\
s_{\pi(I^c)}&\text{if }I\in\mathcal P_3\end{cases}\]
and 
\[y=x+\sum_{I\in\mathcal P_3}s_I.\] 
We claim that $e(y,\tbf)$ has the desired property. We refer the reader to (\ref{eq:symmetric-vertices}) for the $r=4$ case. 
For the general case, note that for $I\in\mathcal P_1\cup\mathcal P_2$ one has $i\in I$ if and only if $\pi(i)\in\pi(I)$. Moreover, for $I\in\mathcal P_3$ one has $i\in I$ if and only if $\pi(i)\not\in\pi(I^c)$. Therefore,
\begin{align*}
e_i(y,\tbf)
&=x+\sum_{I\in\mathcal P_3}s_I+\sum_{I\in\mathcal P_1\cup\mathcal P_2\atop{i\in I}}s_{\pi(I)}+\sum_{I\in\mathcal P_3\atop{i\in I}}s_{\pi(I^c)}\\
&=x+\sum_{I\in\mathcal P_3}s_I+\sum_{I\in\mathcal P_1\cup\mathcal P_2\atop{\pi(i)\in I}}s_{I}+\sum_{I\in\mathcal P_3\atop{\pi(i)\not\in I}}s_{I}\\
&=x+\sum_{I\in\mathcal P_3\atop{\pi(i)\in I}}s_I+\sum_{I\in\mathcal P_1\cup\mathcal P_2\atop{\pi(i)\in I}}s_{I}\\
&=e_{\pi(i)}(x,\s).\qedhere
\end{align*}
\end{proof}

Next we use a similar technique to the proof of Lemma \ref{thm:distinct-tuple} to compute the degree of each $(r-1)$-edge in the hypergraph $\Hsf_{r,t,S}$. Recall that the degree of an $(r-1)$-edge is the number of $r$-edges that it is contained in. Note that in light of Lemma \ref{thm:symmetric-edges} and the discussion preceding it, the degree of the $(r-1)$-edge formed by the set of all coordinates of $e(x,\s)$ but the $k$-th is exactly the number of $(x',\s')\in\F_2^t\times\mathcal T$ such that $e(x,\s)$ and $e(x',\s')$ agree in all but the $k$-th coordinate.

\begin{lemma}
\label{thm:degree}
Recall $\mathcal T$ from (\ref{eq:T}). For $r,t,S$ satisfying Assumption \ref{asmp2}, $(x,\s)\in\F_2^t\times\mathcal T$ and $1\leq k\leq r$, there are exactly\[\left(|S|-(2^{r-1}-2)\right)2^{2^{r-2}-1}\] pairs $(x',\s')\in\F_2^t\times\mathcal T$ such that $e(x,\s)$ and $e(x',\s')$ agree in all but (possibly) the $k$-th coordinate.
\end{lemma}

\begin{proof}
First we construct the desired number of pairs $(x',\s')$. To begin, note that for each $I\in\mathcal P$ exactly one of the following three statements holds:
\begin{enumerate}
    \item there exists a unique $J\in\mathcal P$ such that either $J=I\sqcup\{k\}$ or $I=J\sqcup\{k\}$;\footnote{We use the notation $J=I\sqcup\{k\}$ to imply that the sets $I$ and $\{k\}$ are disjoint and that $J$ is their union.}
    \item there exists $J\in\mathcal P$ such that $I\sqcup J=[r]\setminus\{k\}$;
    \item $I=\{k\}$.
\end{enumerate}
To see this, note that if $k \in I$, then we can let $J = I \setminus \{k\}$ except in the case $I = \{k\}$. On the other hand, if $k \notin I$, we can let $J = I \sqcup \{k\}$ unless $|I| = \lfloor r/2 \rfloor$ or $r$ is even, $k \neq 1$ and $|I| = r/2 - 1$ with $1 \notin I$. But in both these cases, we can take $J$ to be the complement of $I$ in $[r]\setminus\{k\}$.

For $I\in\mathcal P$ satisfying condition 1 and $J\in\mathcal P$ the unique set satisfying either $J=I\sqcup\{k\}$ or $I=J\sqcup\{k\}$, define $(x',\s')$ by\[x'=x\]and\[s'_K=\begin{cases}s_J\qquad&\text{if }K=I,\\s_I&\text{if }K=J,\\s_K&\text{otherwise.}\end{cases}\]Note that $e_i(x,\s)=e_i(x',\s')$ for $i\neq k$.

For $I\in\mathcal P$ satisfying condition 2 and $J\in\mathcal P$ given by $I\sqcup J=[r]\setminus\{k\}$, define $(x',\s')$ by \[x'=x+s_I+s_J\]and\[s'_K=\begin{cases}s_J\qquad&\text{if }K=I,\\s_I&\text{if }K=J,\\s_K&\text{otherwise.}\end{cases}\]Note that $e_i(x,\s)=e_i(x',\s')$ for $i\neq k$. For example, suppose $i\in I$. Then
\begin{align*}
e_i(x',\s')
&=x+s_I+s_J+\sum_{K\in\mathcal P\atop{i\in K}}s'_K\\
&=x+s_I+s_J+s'_I+\sum_{K\in\mathcal P\atop{i\in K\text{ and }K\neq I}}s'_K\\
&=x+s_I+\sum_{K\in\mathcal P\atop{i\in K\text{ and }K\neq I}}s_K\\
&=e_i(x,\s).
\end{align*}
A similar computation covers the complementary case $i\in J$.

For $I=\{k\}$ and any $a \in S$ such that $a\neq s_K$ for all $K\in\mathcal P\setminus\{\{k\}\}$, define $(x',\s')$ by \[x'=x\]and\[s'_K=\begin{cases}a\qquad&\text{if }K=I,\\s_K&\text{otherwise.}\end{cases}\]Note that $e_i(x,\s)=e_i(x',\s')$ for $i\neq k$.

Note that $\mathcal P\setminus\{\{k\}\}$ is partitioned into $2^{r-2}-1$ disjoint pairs $\{I,J\}$ where either $J=I\sqcup\{k\}$, $I=J\sqcup\{k\}$ or $I\sqcup J=[r]\setminus\{k\}$. Using a combination of the three operations described above, we can find $\left(|S|-(2^{r-1}-2)\right)2^{2^{r-2}-1}$ pairs $(x',\s')\in\F_2^t\times\mathcal T$ satisfying $e_i(x,\s)=e_i(x',\s')$ for $i\neq k$: we choose whether or not to swap each of the $2^{r-2}-1$ pairs $\{I,J\}$ and there are $|S|-(2^{r-1}-2)$ choices for the value of $s'_{\{k\}}$.

Now we prove that these are the only $(x',\s')$ with the desired property. Suppose $(x',\s')\in\F_2^t\times\mathcal T$ satisfies $e_i(x,\s)=e_i(x',\s')$ for $i\neq k$. As in the proof of Lemma \ref{thm:distinct-tuple}, this implies
\begin{equation}\label{eq:deg}
\{s_I:I\in\mathcal P\text{ with }|\{i,j\}\cap I|=1\}=\{s'_I:I\in\mathcal P\text{ with }|\{i,j\}\cap I|=1\}
\end{equation} 
for $1\leq i<j\leq r$ with $i\neq k$ and $j\neq k$.

For $I\in\mathcal P$, define $A_I\subset([r]\setminus\{k\})^2$ by
\[A_I=(I\setminus\{k\})\times(I^c\setminus\{k\})\cup(I^c\setminus\{k\})\times(I\setminus\{k\}).\] 
Note that $s_I$ appears on the left-hand side of (\ref{eq:deg}) if and only if $(i,j)\in A_I$. Note that $A_I\neq\emptyset$ for $I\neq\{k\}$. Now partition $\mathcal P\setminus\{\{k\}\}$ into equivalence classes under the relation $I\sim J$ if $A_I=A_J$. The same argument as in the proof of Lemma \ref{thm:distinct-tuple} then implies that for any equivalence class $\Pi\subset\mathcal P\setminus\{\{k\}\}$, we have
\[\{s_I\}_{I\in\Pi}=\{s'_I\}_{I\in\Pi}.\]
Now it is easy to see that these equivalence classes are exactly the pairs $\{I,J\}$ described earlier with $J=I\sqcup\{k\}$, $I=J\sqcup\{k\}$ or $I\sqcup J=[r]\setminus\{k\}$. Therefore, if $e_i(x,\s)=e_i(x',\s')$ for $i\neq k$, then $\s'$ must be one of the $(|S|-(2^{r-1}-2))2^{2^{r-2}-1}$ elements that we have already found. Finally, note that for fixed $i$ and $x,\s,\s'$, there is a unique $x'$ satisfying the equation $e_i(x,\s)=e_i(x',\s')$. Therefore, there are exactly the desired number of pairs $(x',\s')$.
\end{proof}

\subsection{Expansion properties}
\label{ssec:bubble}

Given $r\geq 3$, $t$ and $S\subset\F_2^t$, write $\G_{r,t,S}$ for the graph with vertex set $\F_2^t\times\mathcal T$ and with an edge between $(x,\s)$ and $(y,\tbf)$ if and only if the ordered $r$-tuples $e(x,\s)$ and $e(y,\tbf)$ differ in exactly one coordinate.

For a graph $G$, we use the notation $G^M$ to refer to the multigraph with the same vertex set as $G$ whose edges between $u$ and $v$ correspond to the $M$-step walks in $G$ from $u$ to $v$. We will prove that $\G_{r,t,S}$ is an expander under Assumptions \ref{asmp1}, \ref{asmp2}, \ref{asmp3} by first showing that there is some constant $M=M(r)$ such that $\G_{r,t,S}^M$ contains an expander as a dense subgraph. The specific expander we will find is a multigraph on the vertex set $\F_2^t\times\mathcal T$ with $c$ edges between $(x,\s)$ and $(y,\tbf)$ whenever $xy$ is an edge in $\Cay(\F_2^t,(S+S)\setminus\{0\})$. In the next section, we will show that this result implies that $\G_{r,t,S}$ is an expander.

We prove the claimed combinatorial properties of $\G_{r,t,S}$ in three stages. First we show how to change the value of $s_I$ for some specific $I$ without changing $x$ or too many of the other $s_J$'s. Then we iterate this to walk from $(x,\s)$ to $(x,\tbf)$ for arbitrary $\tbf$. Finally, we combine this walk with our second type of move to walk from $(x,\s)$ to $(y,\tbf)$ for $y$ an arbitrary neighbor of $x$ in the Cayley graph $\Cay(\F_2^t,(S+S)\setminus\{0\})$ and $\tbf$ arbitrary.

\begin{lemma}
\label{thm:single-bubble}
Given $r,t,S$ satisfying Assumption \ref{asmp1}, there is a constant $c=c(r)>0$ such that for each $I\in\mathcal P$, each $(x,\s)\in\F_2^t\times\mathcal T$ and each $a\in S$ distinct from all coordinates of $\s$, there are at least $c|S|^{|I|-1}$ walks of length $|I|$ in $\G_{r,t,S}$ which begin at $(x,\s)$ and end at an element of the form $(x,\s')$, where $s'_I=a$ and $s_J=s'_J$ for all $J\in\mathcal P$ with $J\nsubseteq I$.
\end{lemma}

\begin{proof}
Write $I=\{k_1,\ldots,k_l\}$. Then define $I_i=\{k_1,\ldots,k_i\}$. To walk from $(x,\s)$ to a vertex of the form $(x,\s')$ in $\G_{r,t,S}$ we insert $a$ at $I_1$ and then bubble it up to $I_2,I_3,\ldots,I_l=I$. This finds one such path; to find many paths we insert arbitrary values $a_2,\ldots,a_l$ at the bottom of the lattice on all steps but the first.

Fix $a_2,\ldots,a_l$. Write $\s^0=\s$ and define $\s^1$ by
\begin{equation}
\label{eq:bubble-coordinates-1}
s^1_I=\begin{cases}a\qquad&\text{if }I=I_1=\{k_1\},\\s_I&\text{otherwise.}\end{cases}
\end{equation}
Inductively define $\s^2,\ldots,\s^l$ by
\begin{equation}
\label{eq:bubble-coordinates}
s^i_I=\begin{cases}s_{I_{i-1}}^{i-1}=a\qquad&\text{if }I=I_i,\\s_{I_i}^{i-1}&\text{if }I=I_{i-1},\\a_i&\text{if }I=\{k_i\},\\s^{i-1}_I&\text{otherwise.}\end{cases}
\end{equation}

We claim that $e(x,\s^{i-1})$ and $e(x,\s^i)$ agree everywhere except for the $k_i$-th coordinate. For $i=1$, the $r=4$ case is given in (\ref{eq:move-0}) and for $2\leq i\leq l$, the $r=4$ case is given in (\ref{eq:move-1'}). Consider
\[e_j(x,\s^1)=x+\sum_{I\in \mathcal P\atop{j\in I}}s^1_I.\]
For $j\neq k_1$, the fact that $j\in I$ implies that $I\neq\{k_1\}$, so $s^1_I=s^0_I$. Thus, $e(x,\s^0)$ and $e(x,\s^1)$ agree in all but the $k_1$-th coordinate. Similarly, consider
\[e_j(x,\s^i)=x+\sum_{I\in \mathcal P\atop{j\in I}}s^i_I.\]
For $j\neq k_1,\ldots,k_i$, the fact that $j\in I$ implies that $s^i_I=s^{i-1}_I$. Furthermore, for $j=k_1,\ldots,k_{i-1}$, the sum can be written as
\begin{align*}
e_j(x,\s^i)
&=x+s^i_{I_{i-1}}+s^i_{I_i}+\sum_{I\in \mathcal P\atop{j\in I\text{ and }I\neq I_{i-1},I_i}}s^i_I\\
&=x+s^{i-1}_{I_i}+s^{i-1}_{I_{i-1}}+\sum_{I\in \mathcal P\atop{j\in I\text{ and }I\neq I_{i-1},I_i}}s^{i-1}_I\\
&=e_j(x,\s^{i-1}).
\end{align*}
Thus, $e(x,\s^{i-1})$ and $e(x,\s^i)$ agree in all but the $k_i$-th coordinate.

Finally, note that as long as $a_2,\ldots,a_l$ are distinct from one another and from $a$ and all the coordinates of $\s$, then the tuples $\s^1,\ldots,\s^l$ all lie in $\mathcal T$. Therefore, we have found a walk of length $l=|I|$ from $e(x,\s)$ to $e(x,\s')$ for each valid sequence $a_2,\ldots,a_{|I|}$, of which there are $(|S|-2^{r-1})(|S|-(2^{r-1}+1))\cdots(|S|-(2^{r-1}+|I|-2))\geq c|S|^{|I|-1}$ by Assumption \ref{asmp1}.
\end{proof}

\begin{rem}
\label{thm:single-bubble-restricted}
Suppose we are in the setup of Lemma \ref{thm:single-bubble}, but we are also given a set $X$ not containing $a$ or any of the coordinates of $\s$. Suppose we wish to count walks of length $|I|$ from $(x,\s)$ to an element of the form $(x,\s')$ such that $s'_I=a$ and $s_J=s'_J$ for all $J\in\mathcal P$ with $J\nsubseteq I$ as before, but now with the additional restriction that no coordinate of $\s'$ can lie in $X$. The argument above still covers this case, giving one walk for each choice of $a_2,\ldots,a_{|I|}$ which are distinct from one another and from $a$, all coordinates of $\s$ and all elements of $X$. Thus, there are at least $(|S|-(2^{r-1}+|X|))(|S|-(2^{r-1}+|X|+1))\cdots(|S|-(2^{r-1}+|X|+|I|-2))$ walks. For $X$ small enough, say $|X|\leq 2^r$, Assumption \ref{asmp1} still guarantees at least $c|S|^{|I|-1}$ such walks. 
\end{rem}

\begin{lemma}
\label{thm:full-bubble}
For $r,t,S$ satisfying Assumption \ref{asmp1}, there exist constants $c = c(r)>0$ and $M'=M'(r)$ such that for $x\in\F_2^t$ and $\s,\tbf\in\mathcal T$, there are at least $c|S|^{M'-|\mathcal P|}$ walks of length $M'$ from $(x,\s)$ to $(x,\tbf)$ in $\G_{r,t,S}$.
\end{lemma}

\begin{proof}
Fix any $\ubf \in \mathcal T$ whose set of coordinates are disjoint from those of $\s$ and $\tbf$. Fix an ordering $I_1,I_2,\ldots,I_{|\mathcal P|}$ of $\mathcal P$ so that $I_i\not\subseteq I_j$ for $i<j$. We use the notation \[u\overset{\ell}{\leadsto}v\] to refer to an $\ell$-step walk in $\G_{r,t,S}$ from $u$ to $v$.

Our goal is to walk from $(x,\s)$ to $(x,\ubf)$ by bubbling $u_{I_1}$ up to $I_1$, then bubbling $u_{I_2}$ up to $I_2$, and so on. Lemma \ref{thm:single-bubble} lets us perform each of these individual steps as long as $u_{I_i}$, the element we wish to bubble up, is distinct from all coordinates of the current vertex. Moreover, our choice of ordering means that each successive bubbling does not disturb the previous ones. Thus, it suffices to find many walks of the form
\[(x,\s)\overset{|I_1|}{\leadsto}(x,\ubf^1)\overset{|I_2|}{\leadsto}(x,\ubf^2)\leadsto\cdots\leadsto (x,\ubf),\] 
where $\ubf^i\in\mathcal T$ satisfies $u^i_{I_j}=u_{I_j}$ for all $1\leq j\leq i$ and no coordinate of $\ubf^i$ is equal to $u_{I_j}$ for $j>i$.

For each $1\leq i\leq|\mathcal P|$, note that if $u_{I_i}$ is distinct from all coordinates of $\ubf^{i-1}$, Lemma \ref{thm:single-bubble} guarantees that the number of walks of the form 
\[(x,\ubf^{i-1})\overset{|I_{i}|}{\leadsto}(x,\ubf^{i})\]
is $\Omega(|S|^{|I_{i}|-1})$.\footnote{We use the notation $\Omega(f)$ to denote a quantity of size at least $cf$ where $c=c(r)>0$ is a positive constant that only depends on the uniformity $r$.} Furthermore, Remark \ref{thm:single-bubble-restricted} implies that we can still find this many walks with no coordinate $\ubf^i$ equal to $u_{I_j}$ for $j>i$.

Multiplying these together, we see that there are  $\Omega(|S|^{M''-|\mathcal P|})$ walks of length $M''$ from $(x,\s)$ to $(x,\ubf)$, where $M''=\sum_{I\in\mathcal P}|I|$. Setting $M'=2M''$, we conclude that there are  $\Omega(|S|^{M'-2|\mathcal P|})$ walks of the form 
\[(x,\s)\overset{M''}\leadsto (x,\ubf)\overset{M''}\leadsto (x,\tbf)\] 
for each $\ubf$ whose coordinates are distinct from the coordinates of both $\s$ and $\tbf$. Since there are at least $\Omega(|S|^{|\mathcal P|})$ choices for $\ubf$, summing over them gives the desired result.
\end{proof}

\begin{lemma}
\label{thm:Cay-walks}
For $r,t,S$ satisfying Assumption \ref{asmp1}, there exist constants $c=c(r)>0$ and $M=M(r)$ such that, given $x$ adjacent to $y$ in $\Cay(\F_2^t,(S+S)\setminus\{0\})$ and $\s,\tbf\in\mathcal T$, there are at least $c|S|^{M-|\mathcal P|-2}$ walks of length $M$ from $(x,\s)$ to $(y,\tbf)$ in $\G_{r,t,S}$.
\end{lemma}

\begin{proof}
Define
\[I_1=\left\{3,4,\ldots,\left\lceil\frac r2\right\rceil+1\right\}\]
and
\[I_2=\left\{1,\left\lceil\frac r2\right\rceil+2,\left\lceil\frac r2\right\rceil +3,\ldots,r\right\}.\] 
The relevant properties of these two sets are that $I_1,I_2\in\mathcal P$, they are disjoint and $I_1\cup I_2=[r]\setminus\{2\}$.

Since $x$ is adjacent to $y$ in $\Cay(\F_2^t,(S+S)\setminus\{0\})$, we can write $y=x+a_1+a_2$ for $a_1,a_2\in S$. Suppose now that $\ubf\in\mathcal T$ is such that $u_{I_1}=a_1$ and $u_{I_2}=a_2$ and pick $b\in S$ distinct from all coordinates of $\ubf$. Define $\ubf'$ by
\begin{equation}
\label{eq:cay-walk-coordinates}
u'_I=\begin{cases}u_{I_2}=a_2\qquad&\text{if }I=I_1,\\u_{I_1}=a_1&\text{if }I=I_2,\\b&\text{if }I=\{2\},\\u_I&\text{otherwise.}\end{cases}
\end{equation}

We claim that there are $\Omega(|S|^{2M'-2|\mathcal P|})$ walks of the form
\[(x,\s)\overset{M'}\leadsto (x,\ubf)\overset{1}\leadsto(y,\ubf')\overset{M'}\leadsto (y,\tbf).\]
By Lemma \ref{thm:full-bubble}, there are $\Omega(|S|^{M'-|\mathcal P|})$ walks of length $M'$ from $(x,\s)$ to $(x,\ubf)$ and also from $(y,\ubf')$ to $(y,\tbf)$. All that remains to be checked is that $(x,\ubf)$ and $(y,\ubf')$ are adjacent in $\G_{r,t,S}$. We refer the reader to (\ref{eq:move-2'}) for the $r=4$ example. Now
\[e_i(y,\ubf')=x+a_1+a_2+\sum_{I\in\mathcal P\atop{i\in I}}u'_I.\]
For $i\in I_1$, this can be written as
\begin{align*}
e_i(y,\ubf')
&=x+a_1+a_2+u'_{I_1}+\sum_{I\in\mathcal P\atop{i\in I\text{ and }I\neq I_1}}u'_I\\
&=x+a_1+\sum_{I\in\mathcal P\atop{i\in I\text{ and }I\neq I_1}}u_I\\
&=e_i(x,\ubf).
\end{align*}
A similar computation shows that $e_i(x,\ubf)=e_i(y,\ubf')$ for $i\in I_2$. Therefore, $e(x,\ubf)$ and $e(y,\ubf')$ agree in all but the second coordinate.

Finally, note that there are $\Omega(|S|^{|\mathcal P|-1})$ choices for the pair $(\ubf,\ubf')$ given the constraints that $u_{I_1}=a_1$, $u_{I_2}=a_2$ and $\ubf$ and $\ubf'$ are related by (\ref{eq:cay-walk-coordinates}). Setting $M=2M'+1$ and summing over all choices of $\ubf,\ubf'$ with the above properties, we see that there are $\Omega(|S|^{M-|\mathcal P|-2})$ walks of length $M$ from $(x,\s)$ to $(y,\tbf)$.
\end{proof}

\subsection{Spectral properties}
\label{ssec:spectral}

\begin{lemma}
\label{thm:expander}
For $r\geq 3$, $t$, $\epsilon>0$ and $S\subset\F_2^t$ satisfying Assumptions \ref{asmp1}, \ref{asmp2}, \ref{asmp3}, there exists a constant $c=c(r)>0$ such that $\G_{r,t,S}$ is a $c\epsilon$-expander.
\end{lemma}

\begin{proof}
This follows from Lemma \ref{thm:Cay-walks} in three steps.

First we claim that Assumption \ref{asmp3} implies that $\Cay(\F_2^t,(S+S)\setminus\{0\})$ is an $\epsilon$-expander. Note that 
\[\mathbf A_{\Cay(\F_2^t,(S+S)\setminus\{0\})}=\frac12\left(\mathbf A_{\Cay(\F_2^t,S)}^2-|S|\mathbf I\right).\] 
To see this equation, note that $\Cay(\F_2^t,S)^2$ has an edge from $x$ to $x+s_1+s_2$ for each pair $s_1,s_2\in S$. This differs from $\Cay(\F_2^t,(S+S)\setminus\{0\})$ in that the pairs with $s_1=s_2$ contribute $|S|$ loops at each vertex and every other edge is double-counted since $x+s_1+s_2=x+s_2+s_1$.

Now, by Assumption \ref{asmp3} and the above equation, it follows that the non-trivial eigenvalues of $\Cay(\F_2^t,(S+S)\setminus\{0\})$ lie in $\left[-\frac12|S|,\frac12\left((1-\epsilon)^2|S|^2-|S|\right)\right]$. This interval is contained within $\left[-\frac12(1-\epsilon)(|S|^2-|S|),\frac12(1-\epsilon)(|S|^2-|S|)\right]$, proving that $\Cay(\F_2^t,(S+S)\setminus\{0\})$ is an $\epsilon$-expander.

Define the multigraph $G_{\Cay}$ to have vertex set $\F_2^t\times\mathcal T$ and $c|S|^{M-|\mathcal P|-2}$ edges between $(x,\s)$ and $(y,\tbf)$ whenever $xy$ is an edge in $\Cay(\F_2^t,(S+S)\setminus\{0\})$. Here $c, M$ are the same as in the statement of Lemma \ref{thm:Cay-walks}. Note that $G_{\Cay}$ is regular of degree $d_1\geq c'|S|^M$ for some $c'=c'(r)>0$. We claim that $G_{\Cay}$ is an $\epsilon$-expander. To see this, note that $G_{\Cay}$ is the tensor product of $\Cay(\F_2^t,(S+S)\setminus\{0\})$ with another graph, namely, the complete graph on $|\mathcal T|$ vertices with self-loops where every edge has multiplicity $c|S|^{M-|\mathcal P|-2}$. Then the fact that $\Cay(\F_2^t,(S+S)\setminus\{0\})$ is an $\epsilon$-expander implies the same is true of $G_{\Cay}$.

Second, note that Lemma \ref{thm:degree} implies that $\G_{r,t,S}^M$ is regular of degree $d_2\leq c''|S|^M$ for some $c''=c''(r)$. By Lemma \ref{thm:Cay-walks}, we know that $G_{\Cay}$ is a subgraph of $\G_{r,t,S}^M$. Thus, 
\[\mathbf A_{\G_{r,t,S}^M}=\mathbf A_{G_{\Cay}}+\mathbf A_{\G_{r,t,S}^M\setminus G_{\Cay}}.\] 
Here $\G_{r,t,S}^M\setminus G_{\Cay}$ is a $(d_2-d_1)$-regular graph, which implies that
\[\lambda(\G_{r,t,S}^M)\leq\lambda(G_{\Cay})+(d_2-d_1)\leq\left(1-\epsilon\right)d_1+(d_2-d_1)=\left(1-\frac{d_1}{d_2}\epsilon\right)d_2.\] 
Since $d_1\geq c'|S|^M$ and $d_2\leq c''|S|^M$, we conclude that $\G_{r,t,S}^M$ is a $\frac{c'}{c''}\epsilon$-expander.

Third, the equation
\[\mathbf A_{\G_{r,t,S}^M}=\mathbf A_{\G_{r,t,S}}^M\]
implies that 
\[\lambda(\G_{r,t,S})\leq \left(1-\frac{c'}{c''M}\epsilon\right)d,\]
where $d$ is the degree of $\G_{r,t,S}$.
Since $c',c'',M$ are all constants that depend only on $r$, this proves the desired result.
\end{proof}

Finally, we use this result to prove the main theorem.

\begin{mainthm}{\ref{thm:main}}
For $r\geq 3$, $t$, $\epsilon>0$ and $S\subset\F_2^t$ satisfying Assumptions \ref{asmp1}, \ref{asmp2}, \ref{asmp3}, there exists a constant $c=c(r)>0$ such that the walk graph $\Gw(\Hsf_{r,t,S})$ is a $c\epsilon$-expander.
\end{mainthm}

\begin{proof}
In this proof we consider three graphs: $\G_{r,t,S}$, $\Gw(\Hsf_{r,t,S})$ and a third graph that we will call $\Gw'(\Hsf_{r,t,S})$. The graph $\G_{r,t,S}$ was defined in the previous section; its vertices are pairs $(x,\s)\in\F_2^t\times\mathcal T$. The graph $\Gw(\Hsf_{r,t,S})$ is the walk graph whose vertices are $(r-1)$-edges of $\Hsf_{r,t,S}$ with an edge $ef$ if there is some common $r$-edge that contains both $e$ and $f$. We define the graph $\Gw'(\Hsf_{r,t,S})$ dually to have vertices the $r$-edges of $\Hsf_{r,t,S}$ with an edge $ef$ if there is some common $(r-1)$-edge contained in both $e,f$. By Lemma \ref{thm:degree}, we know that there is some $D$ such that every $(r-1)$-edge of $\Hsf_{r,t,S}$ is contained in exactly $D$ of the $r$-edges. Then $\Gw(\Hsf_{r,t,S})$ is $(r-1)D$-regular and $\Gw'(\Hsf_{r,t,S})$ is $r(D-1)$-regular.

Define a map $\pi\colon V(\G_{r,t,S})\to V(\Gw'(\Hsf_{r,t,S}))$ that sends $(x,\s)$ to the unordered underlying set of $e(x,\s)$. Then this map defines a graph homomorphism $\pi\colon \G_{r,t,S}\to \Gw'(\Hsf_{r,t,S})$. By Lemma \ref{thm:symmetric-edges}, these two graphs have the same degree and $\pi$ is $r!$-to-1 on both the vertex- and edge-sets. We claim that this implies that $\lambda(\Gw'(\Hsf_{r,t,S}))\leq\lambda(\G_{r,t,S})$. For any eigenvector $\vec v\in\R^{V(\Gw'(\Hsf_{r,t,S}))}$, we can lift $\vec v$ to a vector $\pi^{-1}\vec v\in\R^{V(\G_{r,t,S})}$ that is constant on the fibers of $\pi$. The above properties of $\pi$ easily imply that $\pi^{-1}\vec v$ is an eigenvector of $\G_{r,t,S}$ with the same eigenvalue as $\vec v$. Thus, by Lemma \ref{thm:expander}, we conclude that $\lambda(\Gw'(\Hsf_{r,t,S}))\leq (1-c\epsilon)r(D-1)$ for some $c=c(r)$ satisfying $0<c<1/2$. (If the value of $c$ which comes out of Lemma \ref{thm:expander} is not less than $1/2$, we can reduce it to $1/2$ without changing the validity of this eigenvalue bound.) 

Let $\mathbf B$ be the incidence matrix whose rows are indexed by the $(r-1)$-edges of $\Hsf_{r,t,S}$ and whose columns are indexed by the $r$-edges of $\Hsf_{r,t,S}$. Note that the adjacency matrices of $\Gw(\Hsf_{r,t,S})$ and $\Gw'(\Hsf_{r,t,S})$ satisfy 
\[\mathbf A_{\Gw'(\Hsf_{r,t,S})}+r\mathbf I=\mathbf B^T\mathbf B\] 
and
\[\mathbf A_{\Gw(\Hsf_{r,t,S})}+D\mathbf I=\mathbf B\mathbf B^T.\] 
Therefore, the non-trivial eigenvalues of $\mathbf B^T\mathbf B$ lie in the interval $[0,(1-c\epsilon)r(D-1)+r]$. (The lower bound follows since $\mathbf B^T\mathbf B$ is positive semi-definite.) Furthermore, $\mathbf B\mathbf B^T$ has the same eigenvalues as $\mathbf B^T\mathbf B$ with some 0's removed. Therefore, the non-trivial eigenvalues of $\Gw(\Hsf_{r,t,S})$ lie in the interval $[-D,(1-c\epsilon)r(D-1)+r-D]$. Since $r\geq 3$, $c<1/2$ and $D\geq r$, this interval is contained in $[-(1-c\epsilon)(r-1)D,(1-c\epsilon)(r-1)D]$.
\end{proof}

\section{Lower-order random walks}
\label{sec:lower}

Given an $r$-uniform hypergraph $H=(V,E)$, recall that a $k$-edge of $H$ is a $k$-element subset of $V$ that is contained in an $r$-edge of $H$. Two different $k$-edges are said to be adjacent if there is some $(k+1)$-edge of $H$ that they are both contained in. The $k$-th order random walk on $H$ is a random walk that starts at some $k$-edge of $H$ and moves to a random adjacent $k$-edge of $H$ repeatedly.

In the previous section, we showed that the $(r-1)$-st order random walk on $\Hsf_{r,t,S}$ mixes rapidly. In this section, we extend this to prove that the $k$-th order random walk mixes rapidly for all $1\leq k\leq r-1$. We do this by proving that the $k$-th order random walk on $\Hsf_{r,t,S}$ behaves almost exactly the same as the $k$-th order random walk on $\Hsf_{k+1,t,2^{r-k-1}S'}$ where $2^{r-k-1}S'$ is a modification of the $2^{r-k-1}$-fold sumset of $S$ defined below.

In this section, we use $\mathcal P_r$ to refer to $\mathcal P$ of (\ref{eq:P}) and $\mathcal T_r(S)$ to refer to $\mathcal T $ of (\ref{eq:T}) when confusion may arise. Define $lS'=\{s_1+\cdots+s_l:s_1,\ldots,s_l\in S\text{ distinct}\}$. For a hypergraph $H$, we write $\Gw^{(k)}(H)$ for the graph whose vertices are the $k$-edges of $H$ with an edge between each pair of adjacent $k$-edges. 

Let $\G^{(k)}_{r,t,S}$ be the graph whose vertices are the ordered $k$-tuples $(e_i(x,\s))_{i=1}^k$ for $(x,\s)\in\F_2^t\times \mathcal T$. Two vertices of $\G^{(k)}_{r,t,S}$ are adjacent if they agree in all but one coordinate. Note that previously we defined the graph $\G_{r,t,S}$ to have vertex set consisting of pairs $(x,\s)$. This definition agrees with $\G^{(r)}_{r,t,S}$ since for $(x,\s)\neq(y,\tbf)$ the $r$-tuples $e(x,\s)$ and $e(y,\tbf)$ are distinct. However, for $k<r$, there may be many pairs $(x,\s)$ and $(y,\tbf)$ such that $(e_i(x,\s))_{i=1}^k=(e_i(y,\tbf))_{i=1}^k$.

Our proof proceeds in three stages. First we show that $\G^{(r-k)}_{r,t,S}$ is an induced subgraph of $\G_{r-k,t,2^kS'}$. Second, we use this fact to prove that $\G^{(r-k)}_{r,t,S}$ is an expander in the same manner as in Section \ref{sec:proof}. Third, we deduce from this that $\Gw^{(r-k-1)}(\Hsf_{r,t,S})$ is an expander in the same manner as in the proof of Theorem \ref{thm:main}.

Define $\mathcal T'\subseteq\mathcal T_{r-k}(2^kS')$ by
\begin{equation}
    \label{eq:T'}
    \mathcal T'=\left\{\left(\sum_{a=1}^{2^k}s_{a,I}\right)_{I\in\mathcal P_{r-k}}:\text{all the }s_{a,I}\in S\text{ are distinct}\right\}.
\end{equation}
Note that this definition differs from that of $\mathcal T_{r-k}(2^kS')$ in that the latter set is defined by the two conditions that for each $I$ the elements $\{s_{a,I}\}_{1\leq a\leq 2^k}$ are distinct and the elements $\left\{\sum_a s_{a,I}\right\}_{I\in\mathcal P_{r-k}}$ are distinct. Since $S$ satisfies Assumption \ref{asmp2} the condition that defines $\mathcal T_{r-k}(2^kS')$ is a weaker condition than the one that defines $\mathcal T'$.

Now define $U\subseteq V(\G_{r-k,t,2^kS'})$ by 
\begin{equation}
\label{eq:U}
U=\F_2^t\times \mathcal T'.
\end{equation}
We write $\G_{r-k,t,2^kS'}[U]$ for the induced subgraph on vertex set $U$.

Let us briefly discuss the properties of $\G_{r-k,t,2^kS'}$. Assume that $r,t,S,\epsilon$ satisfy Assumptions \ref{asmp1}, \ref{asmp2}, \ref{asmp3}. Obviously, $r-k,t,2^kS'$ still satisfy Assumption \ref{asmp1}. Later on, we will show that $t,2^kS',c\epsilon$ still satisfy Assumption \ref{asmp3} for some constant $c$. However, $r-k,t,2^kS'$ fail Assumption \ref{asmp2}. To see this, suppose $k=1$ and $a,b,c,d\in S$. Then $a+b,a+c,b+d,c+d\in 2^kS'$, but $(a+b)+(c+d)=(a+c)+(b+d)$. This means that Lemmas \ref{thm:distinct-tuple} and \ref{thm:degree} may not apply to $\G_{r-k,t,2^kS'}$. However, we will see that these lemmas still apply to $\G_{r-k,t,2^kS'}[U]$.

\begin{lemma}
\label{thm:distinct-tuple-lower}
Recall $\mathcal T'$ from (\ref{eq:T'}). For $r,t,S$ satisfying Assumption \ref{asmp2} and $0\leq k\leq r-2$, the ordered $(r-k)$-tuples $e(x,\s)$ for $x\in \F_2^t$ and $\s\in\mathcal T'$ are all distinct.
\end{lemma}

\begin{proof}
This follows from essentially the same proof as Lemma \ref{thm:distinct-tuple}.

Suppose $e(x,\s)=e(x',\s')$ for $x,x'\in\F_2^t$ and $\s,\s'\in\mathcal T'$, where $s_I=\sum_{a=1}^{2^k}s_{a,I}$ and $s'_I=\sum_{a=1}^{2^k}s'_{a,I}$ with $s_{a,I},s'_{a,I}\in S$. Then we know that \[e_i(x,\s)+e_j(x,\s)=e_i(x',\s')+e_j(x',\s')\] for $1\leq i<j\leq r-k$. This can be rewritten as
\[\sum_{I\in\mathcal P_{r-k}\atop{|\{i,j\}\cap I|=1}}\sum_{a=1}^{2^k}s_{a,I}=\sum_{I\in\mathcal P_{r-k}\atop{|\{i,j\}\cap I|=1}}\sum_{a=1}^{2^k}s'_{a,I}.\]
Since the $s_{a,I}$'s are all distinct, the $s'_{a,I}$'s are all distinct and $S$ satisfies Assumption \ref{asmp2}, the same argument as in Lemma \ref{thm:distinct-tuple} implies that $\{s_{a,I}\}_a=\{s'_{a,I}\}_a$ for each $I\in\mathcal P_{r-k}$. Thus $\s=\s'$. Finally, since $e_1(x,\s)=e_1(x',\s')$, this implies that $x=x'$, as desired.
\end{proof}

\begin{lemma}
\label{thm:isom}
Recall $U$ from (\ref{eq:U}). For $r,t,S$ satisfying Assumption \ref{asmp2} and $0\leq k\leq r-2$, the graph $\G^{(r-k)}_{r,t,S}$ is isomorphic to $\G_{r-k,t,2^kS'}[U]$.
\end{lemma}

\begin{proof}
We will define maps $\psi\colon V(\G_{r,t,S})\to V(\G^{(r-k)}_{r,t,S})$ and $\phi\colon V(\G_{r,t,S})\to U\subseteq V(\G_{r-k,t,2^kS'})$ and show that $\phi\circ\psi^{-1}$ is a well-defined isomorphism between $\G^{(r-k)}_{r,t,S}$ and $\G_{r-k,t,2^kS'}[U]$.

Define $\psi\colon V(\G_{r,t,S})\to V(\G^{(r-k)}_{r,t,S})$ by $\psi((x,\s))=(e_i(x,\s))_{i=1}^{r-k}$.

Next start by defining a map $\phi\colon\mathcal P_r\to\mathcal P_{r-k}\cup\{\emptyset\}$. For $J\in\mathcal P_r$, note that exactly one of $J\cap[r-k]$ and $[r-k]\setminus J$ lies in $\mathcal P_{r-k}\cup\{\emptyset\}$. Let $\phi(J)$ be the one of these two sets that lies in $\mathcal P_{r-k}\cup\{\emptyset\}$.
Note that for any $I \in \mathcal P_{r-k}\cup\{\emptyset\}$, there are exactly $2^k$ sets $J \in \mathcal P_{r}$ such that $\phi(J) = I$.
Then $\phi\colon V(\G_{r,t,S})\to V(\G_{r-k,t,2^kS'})$ is given by $\phi((x,\s))=(y,\tbf)$ where
\[t_I=\sum_{J\in\mathcal P_r\atop{\phi(J)=I}}s_J\]
and
\[y=x+\sum_{J\in\mathcal P_r\atop{J\cap[r-k]\not\in\mathcal P_{r-k}\cup\{\emptyset\}}}s_J.\] 
This map has two important properties. First, $\phi((x,\s))\in U\subseteq V(\G_{r-k,t,2^kS'})$ for any $(x,\s)\in V(\G_{r,t,S})$. Second, if $\phi((x,\s))=(y,\tbf)$ then $e_i(x,\s)=e_i(y,\tbf)$ for $1\leq i\leq k-r$. To see this property, write $\mathcal P_r=\mathcal P_r^{(1)}\sqcup\mathcal P_r^{(2)}$, where $\mathcal P_r^{(1)}$ consists of those $J$ such that $J\cap[r-k]\in\mathcal P_{r-k}\cup\{\emptyset\}$ and $\mathcal P_r^{(2)}$ consists of the other $J$. Then note that
\begin{align*}
e_i(y,\tbf)
&=x+\sum_{J\in\mathcal P_r^{(2)}}s_J
+\sum_{I\in\mathcal P_{r-k}\atop{i\in I}}\sum_{J\in\mathcal P_r\atop{\phi(J)=I}}s_J\\
&=x+\sum_{J\in\mathcal P_r^{(2)}}s_J
+\sum_{J\in\mathcal P_r^{(1)}\atop{i\in J}}s_J
+\sum_{J\in\mathcal P_r^{(2)}\atop{i\not\in J}}s_J\\
&=e_i(x,\s).
\end{align*}
By Lemma \ref{thm:distinct-tuple-lower}, this implies that $\phi((x,\s))=\phi((x',\s'))$ if and only if $e_i(x,\s)=e_i(x',\s')$ for all $1\leq i\leq r-k$. This is the same condition as the one for $\psi((x,\s))=\psi((x,\s'))$. Therefore, we see that $\phi\circ\psi^{-1}\colon V(\G^{(r-k)}_{r,t,S})\to U$ is a well-defined map that sends $(v_1,\ldots,v_{r-k})$ to $(y,\tbf)\in U$ such that $(e_i(y,\tbf))_{i=1}^{r-k}=(v_1,\ldots,v_{r-k})$. Therefore, $\phi\circ\psi^{-1}$ is a graph isomorphism, as desired.
\end{proof}

\begin{lemma}
\label{thm:degree-lower}
Recall $U$ from (\ref{eq:U}). For $r,t,S$ satisfying Assumption \ref{asmp2} and $0\leq k\leq r-2$, there exist constants $c=c(r)$ and $D=D(r,k,|S|)$ such that the induced subgraph $\G_{r-k,t,2^kS'}[U]$ is $D$-regular with 
\[D\leq c|2^kS'|\leq c|S|^{2^k}.\]
\end{lemma}

\begin{proof}
This follows from a similar argument to Lemma \ref{thm:degree}.

Fix $(x,\s)\in\F_2^t\times\mathcal T'$ and $1\leq l\leq r-k$. We wish to count pairs $(x',\s')\in\F_2^t\times\mathcal T'$ such that $e_i(x,\s)=e_i(x',\s')$ for all $i\neq l$. Say $\s,\s'$ are defined by $s_I=\sum_{a=1}^{2^k}s_{a,I}$ and $s'_I=\sum_{a=1}^{2^k}s'_{a,I}$. We know that\[\sum_{I\in\mathcal P_{r-k}\atop{|\{i,j\}\cap I|=1}}\sum_{a=1}^{2^k}s_{a,I}=\sum_{I\in\mathcal P_{r-k}\atop{|\{i,j\}\cap I|=1}}\sum_{a=1}^{2^k}s'_{a,I}\] for $1\leq i<j\leq r-k$ with $i\neq l$ and $j\neq l$.

As in the proof of Lemma \ref{thm:degree}, we partition $\mathcal P_{r-k}\setminus\{\{l\}\}$ into pairs $\{I,J\}$ where either $J=I\sqcup\{l\}$, $I=J\sqcup\{l\}$ or $I\sqcup J=[r-k]\setminus\{l\}$. Using essentially the same argument as in the proof of Lemma \ref{thm:degree}, we conclude that for each such pair $\{I,J\}$,
\[\{s_{a,I}\}_a\cup\{s_{a,J}\}_a=\{s'_{a,I}\}_a\cup\{s'_{a,J}\}_a.\]

Therefore, for each $(x,\s)\in\F_2^t\times\mathcal T'$ and $1\leq l\leq r-k$ the number of pairs $(x',\s')\in\F_2^t\times\mathcal T'$ satisfying $e_i(x,\s)=e_i(x',\s')$ for all $i\neq l$ is exactly equal to\[\binom{|S|-(2^{r-1}-2^{k+1})}{2^k}\binom{2^{k+1}}{2^k}^{2^{r-k-2}-1}.\] The first term is the number of choices for $s'_{\{l\}}=\sum_{a=1}^{2^k}s'_{a,\{l\}}$ where the $s'_{a,\{l\}}$ are distinct from the $s_{a,I}$ for $I\neq \{l\}$. The second term is the number of ways to choose $s'_I=\sum_{a=1}^{2^k}s'_{a,I}$ and $s'_J=\sum_{a=1}^{2^k}s'_{a,J}$ given the set $\{s'_{a,I}\}_a\cup\{s'_{a,J}\}_a$ for each of the $2^{r-k-2}-1$ pairs $\{I,J\}$ that $\mathcal P_{r-k}\setminus\{\{l\}\}$ is partitioned into.
\end{proof}

\begin{lemma}
\label{thm:asmp3-lower}
For $r,t,\epsilon,S$ satisfying Assumptions \ref{asmp1}, \ref{asmp3} and $0\leq k\leq r-2$, there exists a constant $c>0$ such that $\Cay(\F_2^t,2^{k+1}S')$ is a $c\epsilon$-expander.
\end{lemma}

\begin{proof}
For brevity, write $m=2^{k+1}$ and $|S|=d$. Let $\mathbf A$ denote the adjacency matrix of $\Cay(\F_2^t,S)$ and let its eigenvalues be $d=\lambda_1\geq|\lambda_2|\geq\cdots\geq|\lambda_N|$.

Observe that the edges in $\Cay(\F_2^t,S)^m$ incident to vertex $x$ are in bijection with $m$-tuples $(s_1,\ldots,s_m)\in S^m$, i.e., for each tuple there is an edge $(x,x+s_1+\cdots+s_m)$. Similarly, note that the edges in $m!\cdot\Cay(\F_2^t,mS')$ are in bijection with the tuples $(s_1,\ldots,s_m)$ where all $m$ coordinates are distinct.

For a partition $P$ of $[m]$, say that a tuple $(s_1,\ldots,s_m)$ is of type $P$ if whenever $i,j$ are in the same part of $P$ then $s_i=s_j$. Write $A_P(S)$ for the multiset with element $s_1+\cdots+s_m$ for each $(s_1,\ldots,s_m)$ of type $P$.

By M\"obius inversion, we know that
\[\mathbf A_{m!\cdot\Cay(\F_2^t,mS')}=\sum_{P}\mu(P)\mathbf A_{\Cay(\F_2^t,A_P(S))},\]
where $\mu(P)$ is the M\"obius function $(-1)^{m-|P|}\prod_{p\in P}(|p|-1)!$. For a partition $P$, let $e(P)$ be the number of even-sized parts and $o(P)$ be the number of odd-sized parts. It is not hard to see that
\[\mathbf A_{\Cay(\F_2^t,A_P(S))}=d^{e(P)}\mathbf A^{o(P)}.\]
Therefore, we conclude that
\[\mathbf A_{\Cay(\F_2^t,mS')}=\frac1{m!}\sum_{P}\mu(P)d^{e(P)}\mathbf A^{o(P)}.\]

Write \[f(x)=\frac1{m!}\sum_{P}\mu(P)d^{|P|}x^{o(P)}.\] Note that since $m=2^{k+1}$ is even, any partition of $[m]$ has an even number of odd-sized parts. Thus, $f(x)$ is an even function. Therefore, the second largest eigenvalue of $\mathcal A_{\Cay(\F_2^t,mS')}$ is bounded in absolute value by $\sup_{0\leq x\leq1-\epsilon}|f(x)|$. We can upper bound $f(x)$ by
\begin{equation}
\label{eq:f-upper}
\begin{split}
f(x)&\leq (1-\epsilon/2)\frac1{m!}\sum_{P}\mu(P)d^{|P|}-(1-\epsilon/2-x^m)\frac{d^m}{m!}\\
&\qquad+\frac1{m!}\sum_{P\neq(1,\ldots,1)}|\mu(P)|d^{|P|}\left|1-\epsilon/2-x^{o(P)}\right|.
\end{split}\end{equation}
The first term is exactly $(1-\epsilon/2)\binom dm$. To bound the last term, we use the following inequality.

\begin{claim}
\label{thm:mobius-sum}
\[\sum_{P\vdash [m]\atop{|P|=m-a}}|\mu(P)|\leq m^{2a}.\]
\end{claim}

\begin{proof}
We use $|\mu(P)|=\prod_{p\in P}(|p|-1)!\leq\prod_{p\in P}|p|!$. Now
\[\sum_{P\vdash [m]\atop{|P|=m-a}}\prod_{p\in P}|p|!\]
has a combinatorial interpretation as the number of partitions of $[m]$ into $m-a$ parts where the elements of each part are ordered. Such a partition can be represented schematically as a directed graph on $m$ labeled vertices where each connected component is a directed path. The condition that the partition has $m-a$ parts is equivalent to the constraint that the graph has $a$ edges. But there are $m(m-1)$ choices for a single edge, so $(m(m-1))^a\leq m^{2a}$ is a simple upper bound on the number of $m$-vertex $a$-edge directed graphs.
\end{proof}

We are interested in the case where $m=2^{k+1} \leq 2^{r-1}$ which, with Assumption~\ref{asmp1}, implies that $d\geq 4m^2$. Together with Claim \ref{thm:mobius-sum} this implies that \[\sum_{P\neq(1,\ldots,1)}|\mu(P)|d^{|P|}=\sum_{a>1}\sum_{|P|=m-a}|\mu(P)|d^{m-a}\leq d^m\sum_{a>1}\left(\frac{m^2}d\right)^a\leq 2m^2d^{m-1}.\]
Using this with (\ref{eq:f-upper}) gives
\[\begin{split}
f(x)&\leq(1-\epsilon/2)\binom dm-(1-\epsilon/2-x^m)\frac{d^m}{m!}\\
&\qquad+\max\{\epsilon/2,1-\epsilon/2-x^m\}\frac{2m^2d^{m-1}}{m!}
\end{split}\]for $0\leq x\leq 1-\epsilon$, since $\left|1-\epsilon/2-x^{o(P)}\right|$ is extremized when $o(P)=0,m$. But
\[\max\{\epsilon/2,1-\epsilon/2-x^m\}\frac{2m^2d^{m-1}}{m!}\leq(1-\epsilon/2-x^m)\frac{d^m}{m!},\]
since $1-\epsilon/2-x^m\geq1-\epsilon/2-(1-\epsilon)=\epsilon/2$ and $\frac{2m^2d^{m-1}}{m!}\leq\frac{d^m}{m!}$. Therefore, \[f(x)\leq (1-\epsilon/2)\binom dm.\]

For the lower bound, we use the analogous inequality
\[\begin{split}
f(x)&\geq (1-\epsilon/2)\binom dm-(1-\epsilon/2-x^m)\frac{d^m}{m!}\\
&\qquad-\max\{\epsilon/2,1-\epsilon/2-x^m\}\frac{2m^2d^{m-1}}{m!}.
\end{split}\]
Since 
\[\max\{\epsilon/2,1-\epsilon/2-x^m\}\frac{2m^2d^{m-1}}{m!}\leq\frac12(1-\epsilon/2-x^m)\frac{d^m}{m!},\]
we conclude
\[f(x)\geq(1-\epsilon/2)\binom dm-\frac32(1-\epsilon)\frac{d^m}{m!}.\]
Since $d\geq 4m^2$, we have
\[\binom dm\geq \left(\frac{d-m}{d}\right)^m\frac{d^m}{m!}\geq\left(1-\frac1{4m}\right)^m\frac{d^m}{m!} \geq e^{-2/7} d^m/m!.\]
Combining these inequalities yields 
\[f(x)\geq\left(1-\frac32e^{2/7}\right)(1-\epsilon/2)\binom dm\approx -0.996(1-\epsilon/2)\binom dm,\]
as required.
\end{proof}

Note that Lemma \ref{thm:Cay-walks} only requires Assumption \ref{asmp1}. Therefore, if we define $G_{\Cay}$ to be the multigraph with vertex set $\F_2^t\times\mathcal T_{r-k}(2^kS')$ and $c|2^kS'|^{M-|\mathcal P_{r-k}|-2}$ edges between $(x,\s)$ and $(y,\tbf)$ if $xy$ is an edge in $\Cay(\F_2^t,(2^kS'+2^kS')\setminus\{0\})$, then we know that $\G_{r-k,t,2^kS'}^M$ contains $G_{\Cay}$ as a subgraph. Furthermore, it (essentially) follows from Lemma \ref{thm:asmp3-lower} that $G_{\Cay}$ is an expander. However, this does not imply that $\G_{r-k,t,2^kS'}$ is an expander since we do not have a degree bound.

Lemma \ref{thm:degree-lower} gives us the required degree bound on $\G_{r-k,t,2^kS'}[U]$, so to prove that this subgraph is an expander all we need to do is prove that $G_{\Cay}[U]$ is a subgraph of $\G_{r-k,t,2^kS'}[U]^M$. This follows by the same technique as in the proof of Lemma \ref{thm:Cay-walks} with some small modifications.

\begin{lemma}
\label{thm:single-bubble-lower}
Given $r,t,S$ satisfying Assumption \ref{asmp1} and $0\leq k\leq r-2$, there exists a constant $c=c(r)>0$ such that for each $I\in\mathcal P_{r-k}$, each $(x,\s)\in\F_2^t\times\mathcal T'$ and each $a\in 2^kS'$ disjoint from all coordinates of $\s$, there are at least $c|2^kS'|^{|I|-1}$ walks of length $|I|$ in $\G_{r-k,t,2^kS'}[U]$ which begin at $(x,\s)$ and end at an element of the form $(x,\s')$ where $s'_I=a$ and $s_J=s'_J$ for all $J\in\mathcal P_{r-k}$ with $J\nsubseteq I$.
\end{lemma}

\begin{proof}
Following the proof of Lemma \ref{thm:single-bubble}, write $|I|=l$ and note that for each sequence $a_2,\ldots,a_l\in 2^kS'$ we can define $\s=\s^0,\s^1,\ldots,\s^l$ by (\ref{eq:bubble-coordinates-1}) and (\ref{eq:bubble-coordinates}) such that each pair $e(x,\s^i)$ and $e(x,\s^{i+1})$ agree in all but one coordinate.

These are exactly the desired paths as long as $\s^1,\ldots,\s^l$ lie in $\mathcal T'$. To satisfy this condition we simply need to pick $a_2,\ldots,a_l$ disjoint from each other and from $a$ and the coordinates of $\s$. The number of ways to do this \[\binom{|S|-2^{r-1}}{2^k}\binom{|S|-2^{r-1}-2^k}{2^k}\cdots\binom{|S|-2^{r-1}-(l-2)2^k}{2^k}\geq c|2^kS'|^{l-1}.\]The inequality follows by Assumption \ref{asmp1}.
\end{proof}

\begin{lemma}
\label{thm:full-bubble-lower}
Given $r,t,S$ satisfying Assumption \ref{asmp1} and $0\leq k\leq r-2$, there exist constants $c=c(r)>0$ and $M'=M'(r-k)$ such that for $x\in \F_2^t$ and $\s,\tbf\in\mathcal T'$, there are at least $c|2^kS'|^{M'-|\mathcal P_{r-k}|}$ walks of length $M'$ from $(x,\s)$ to $(x,\tbf)$ in $\G_{r-k,t,2^kS'}[U]$.
\end{lemma}

\begin{proof}
Following the proof of Lemma \ref{thm:full-bubble}, fix $\ubf\in\mathcal T'$ all of whose coordinates are disjoint from all the coordinates of both $\s$ and $\tbf$. By exactly the same argument, we can repeatedly use Lemma \ref{thm:single-bubble-lower} to find $c|2^kS'|^{M'-2|\mathcal P_{r-k}|}$ walks of length $M'$ from $(x,\s)$ to $(x,\ubf)$ to $(x,\tbf)$. Finally, note that there are
\[\begin{split}\binom{|S|-(2^r-2^{k+1})}{2^k}&\binom{|S|-(2^r-2^{k+1})-2^k}{2^k}\\&\cdots\binom{|S|-(2^r-2^{k+1})-(|\mathcal P_{r-k}|-1)2^k}{2^k}\geq c'|2^kS'|^{|\mathcal P_{r-k}|}\end{split}\]
choices for $\ubf$, completing the proof.
\end{proof}

\begin{lemma}
\label{thm:Cay-walks-lower}
Given $r,t,S$ satisfying Assumption \ref{asmp1} and $0\leq k\leq r-2$, there exist constants $c=c(r)>0$ and $M=M(r-k)$ such that, given $x$ adjacent to $y$ in $\Cay(\F_2^t,2^{k+1}S')$ and $\s,\tbf\in\mathcal T'$, there are at least $c|2^kS'|^{M-|\mathcal P_{r-k}|-2}$ walks of length $M$ from $(x,\s)$ to $(y,\tbf)$ in $\G_{r-k,t,2^kS'}[U]$.
\end{lemma}

\begin{proof}
Following the proof of Lemma \ref{thm:Cay-walks}, let $I_1,I_2$ be disjoint elements of $\mathcal P_{r-k}$ such that $I_1\sqcup I_2=[r-k]\setminus\{2\}$. Now since $x$ is adjacent to $y$ in $\Cay(\F_2^t,2^{k+1}S')$, there exist $a_1,a_2\in 2^kS'$ such that $y=x+a_1+a_2$. For any $\ubf\in T'$ satisfying $u_{I_1}=a_1$ and $u_{I_2}=a_2$ and any $b$ disjoint from all coordinates of $\ubf$, define $\ubf'$ as in (\ref{eq:cay-walk-coordinates}). Then, by the exact same argument, we can use Lemma \ref{thm:full-bubble-lower} twice to find $c|2^kS'|^{2M'-2|\mathcal P_{r-k}|}$ walks of length $M=2M'+1$ from $(x,\s)$ to $(x,\ubf)$ to $(y,\ubf')$ to $(y,\tbf)$. Finally, note that there are
\[\binom{|S|-2\cdot2^k}{2^k}\binom{|S|-3\cdot2^k}{2^k}\cdots\binom{|S|-|\mathcal P_{r-k}|\cdot2^k}{2^k}\geq c'|2^kS'|^{|\mathcal P_{r-k}|-1}\] 
choices for $(\ubf,\ubf')$, completing the proof.
\end{proof}

\begin{lemma}
\label{thm:expander-lower}
For $r\geq 3$, $t$, $\epsilon>0$ and $S\subset\F_2^t$ satisfying Assumptions \ref{asmp1}, \ref{asmp2}, \ref{asmp3} and $0\leq k\leq r-2$, there exists a constant $c=c(r)>0$ such that $\G^{(r-k)}_{r,t,S}$ is a $c\epsilon$-expander.
\end{lemma}

\begin{proof}
Following the proof of Lemma \ref{thm:expander}, define $G_{\Cay}$ to be the multigraph with vertex set $U$ and  $c|2^kS'|^{M-|\mathcal P_{r-k}|-2}$ edges between $(x,\s)$ and $(y,\tbf)$ if $xy$ is an edge in $\Cay(\F_2^t,2^{k+1}S')$. Here $c,M$ are the same constants as in the proof of Lemma \ref{thm:Cay-walks-lower}. Note that $G_{\Cay}$ is regular of degree $d_1\geq c'|2^kS'|^M$ for some $c'=c'(r)>0$. Furthermore, by Lemma \ref{thm:asmp3-lower}, we know that $G_{\Cay}$ is a $c^* \epsilon$-expander.

By Lemma \ref{thm:degree-lower}, we know that $\G_{r-k,t,2^kS'}[U]^M$ is regular of degree $d_2\leq c''|2^kS'|^M$ for some $c''=c(r)$ and, by Lemma \ref{thm:Cay-walks-lower}, we know that $G_{\Cay}$ is a subgraph of $\G_{r-k,t,2^kS'}[U]^M$. By the exact same proof as Lemma \ref{thm:expander}, this implies that $\G_{r-k,t,2^kS'}[U]$ is a $c_0 \epsilon$-expander for some $c_0$ depending on $c', c'', c^*$ and $M$. Since this graph is isomorphic to $\G^{(r-k)}_{r-k,t,S}$, this proves the desired result.
\end{proof}

\begin{mainthm}{\ref{thm:main-lower}}
For $r\geq 3$, $t$, $\epsilon>0$ and $S\subset\F_2^t$ satisfying Assumptions \ref{asmp1}, \ref{asmp2}, \ref{asmp3}, there exists a constant $c=c(r)>0$ such that for every $1\leq k\leq r-1$, the $k$-th order walk graph $\Gw^{(k)}(\Hsf_{r,t,S})$ is a $c\epsilon$-expander.
\end{mainthm}

\begin{proof}
First assume that $k>1$. This case is essentially the same as the proof of Theorem \ref{thm:main} from Lemma \ref{thm:expander}.

By Lemma \ref{thm:degree-lower}, we know that there is some $D_k$ such that every $k$-edge of $\Hsf_{r,t,S}$ is contained in exactly $D_k$ of the $(k+1)$-edges.

As before, we consider three graphs: $\G^{(k+1)}_{r,t,S}$, $\Gw'^{(k+1)}(\Hsf_{r,t,S})$, $\Gw^{(k)}(\Hsf_{r,t,S})$. The vertices of $\G^{(k+1)}_{r,t,S}$ are ordered $(k+1)$-tuples $(v_1,\ldots,v_{k+1})$ such that there exists $(x,\s)\in\F_2^t\times\mathcal T$ with $v_i=e_i(x,\s)$. We define $\Gw'^{(k+1)}(\Hsf_{r,t,S})$ to be the graph whose vertices are the unordered $(k+1)$-edges of $\Hsf_{r,t,S}$ where two $(k+1)$-edges are adjacent if they contain a common $k$-edge. We know that the first two of these graphs are $(k+1)(D_k-1)$-regular, while the third graph is $kD_k$-regular.

Now define a graph homomorphism $\pi\colon \G^{(k+1)}_{r,t,S}\to\Gw'^{(k+1)}(\Hsf_{r,t,S})$ that sends a vertex $(v_1,\ldots,v_{k+1})$ to the vertex $\{v_1,\ldots,v_{k+1}\}$. These two graphs are regular of the same degree and, by Lemma \ref{thm:symmetric-edges} and the discussion preceding it, the preimage of each vertex has size $(k+1)!$. As in the proof of Theorem~\ref{thm:main}, it follows that $\lambda(\Gw'^{(k+1)}(\Hsf_{r,t,S}))\leq\lambda(\G^{(k+1)}_{r,t,S})\leq(1-c\epsilon)(k+1)(D_k-1)$.

Next define $\mathbf B$ to be the incidence matrix whose rows are indexed by the $k$-edges of $\Hsf_{r,t,S}$ and whose columns are indexed by the $(k+1)$-edges of $\Hsf_{r,t,S}$. We know that
\[\mathbf A_{\Gw'^{(k+1)}(\Hsf_{r,t,S})}+(k+1)\mathbf I=\mathbf B^T\mathbf B\]
and
\[\mathbf A_{\Gw^{(k)}(\Hsf_{r,t,S})}+D_k\mathbf I=\mathbf B\mathbf B^T.\]
From the first equation, we conclude that the non-trivial eigenvalues of $\mathbf B^T\mathbf B$ lie in the interval $[0,(1-c\epsilon)(k+1)(D_k-1)+k+1]$, so the non-trivial eigenvalues of $\Gw^{(k)}(\Hsf_{r,t,S})$ lie in the interval $[-D_k,(1-c\epsilon)(k+1)(D_k-1)+k+1-D_k]$. Since $k\geq 2$, $D_k\geq k+1$ and choosing $c$ such that $c\epsilon<1/2$, this interval is contained in $[-(1-c\epsilon)kD_k,(1-c\epsilon)kD_k]$.

Now assume that $k=1$. Note that a first order random walk on $\Hsf_{r,t,S}$ is exactly the same as an ordinary random walk on the graph $\Cay(\F_2^t,2^{r-2}S')$. By Lemma \ref{thm:asmp3-lower}, we know that $\Cay(\F_2^t,2^{r-2}S')$ is a $c\epsilon$-expander, completing the proof.
\end{proof}

\section{A discrepancy result}
\label{sec:discrepancy}

For $H=(V,E)$ an $r$-uniform hypergraph and $V_1,\ldots,V_r\subseteq V$, let $e_H(V_1,\ldots,V_r)$ be the number of tuples $(v_1,\ldots,v_r)\in V_1\times\cdots\times V_r$ such that $\{v_1,\ldots,v_r\}\in E$. In this self-contained section, we prove that our hypergraph $\Hsf_{r,t,S}$ satisfies the pseudorandomness condition that $e_{\Hsf_{r,t,S}}(V_1,\ldots,V_r)$ is always close to its expected value. For simplicity, we write $e(V_1,\ldots,V_r)$ for $e_{\Hsf_{r,t,S}}(V_1,\ldots,V_r)$ throughout this section.

A general result of this type was proven by Parzanchevski \cite{P17}, saying that the desired pseudorandomness condition holds under essentially the hypothesis on the spectra of all the adjacency matrices that we proved in Theorem \ref{thm:expander-lower}. However, Parzanchevski's result only applies when $\frac{|\lambda_2|}{\lambda_1}<\epsilon_0(r)$ where $\epsilon_0(r)$ is a small constant depending on the uniformity, while we have only proved that $\frac{|\lambda_2|}{\lambda_1}<1-c(r)\epsilon$ where $\epsilon$ is the expansion parameter of the original Cayley graph and $c(r)$ is a very small constant depending on the uniformity. We will therefore use a different proof method.

In this section, $\mathcal Q$ will denote an arbitrary multiset supported in $2^{[r]}$. For such a $\mathcal Q$ and $(x,\s)\in\F_2^t\times S^{\mathcal Q}$, define 
\[e_{\mathcal Q,i}(X,\s)=x+\sum_{I\in\mathcal Q\atop{i\in I}}s_I.\] 
Write $e_{\mathcal Q}(x,\s)$ for the $r$-tuple $(e_{\mathcal Q,1}(x,\s),\ldots,e_{\mathcal Q,r}(x,\s))$.

We wish to compute $e(V_1,\ldots,V_r)$, the number of pairs $(x,\s)\in\F_2^t\times\mathcal T$ such that $e(x,\s)\in V_1\times\cdots\times V_r$. We will start by computing $f_{\mathcal Q}(V_1,\ldots,V_r)$, defined to be the number of pairs $(x,\s)\in\F_2^t\times S^{\mathcal Q}$ such that $e_{\mathcal Q}(x,\s)\in V_1\times\cdots\times V_r$. We will then use M\"obius inversion to turn this into a formula for the desired quantity.

\begin{lemma}
\label{thm:disc-with-rep}
For $r\geq3$, $t$, $\epsilon>0$ and $S\subset \F_2^t$ satisfying Assumptions \ref{asmp2}, \ref{asmp3}, $\mathcal Q$ a multiset supported in $2^{[r]}$ and $V_1,\ldots,V_r\subseteq\F_2^t$, 
\[\left|f_{\mathcal Q}(V_1,\ldots,V_r)-\frac{|S|^{|\mathcal Q|}}{(2^t)^{r-1}}|V_1|\cdots|V_r|\right|\leq(1-\epsilon)(r-1)|S|^{|\mathcal Q|}\sqrt{\lvert V_1\rvert\lvert V_r\rvert}.\]
\end{lemma}

\begin{proof}
For $1\leq k\leq r$, define $\mathcal Q^{(k)}\subseteq\mathcal Q$ by
\begin{equation*}
    \mathcal Q^{(k)}=\{I\in\mathcal Q:0<|I\cap\{1,\ldots,k\}|<k\}.
\end{equation*}
Furthermore, write $e^{(k)}_{\mathcal Q}(x,\s)$ for the $k$-tuple $(e_{\mathcal Q^{(k)},1}(x,\s),\ldots,e_{\mathcal Q^{(k)},k}(x,\s))$.
Let $F^k(V_1,\ldots,V_k)$ be the set of pairs $(y,\s)\in\F_2^t\times S^{\mathcal Q^{(k)}}$ such that $e^{(k)}_{\mathcal Q}(y,\s)\in V_1\times\cdots\times V_k$. We write $f^k(V_1,\ldots,V_k)=|F^k(V_1,\ldots,V_k)|$.

Define $L^{(k)}=\{I\in\mathcal Q:I\cap\{1,\ldots,k+1\} = \{k+1\}\}$ and $U^{(k)}=\{I\in\mathcal Q:I\cap\{1,\ldots,k+1\}=\{1,\ldots,k\}\}$. Then
\[\mathcal Q^{(k+1)}=\mathcal Q^{(k)}\sqcup\mathcal L^{(k)}\sqcup\mathcal U^{(k)}.\]
By definition, $F^{k+1}(V_1,\ldots,V_{k+1})$ can therefore be written as the set of pairs $(x,(\s,\tbf,\ubf))\in\F_2^t\times \left(S^{\mathcal Q^{(k)}}\times S^{\mathcal L^{(k)}}\times S^{\mathcal U^{(k)}}\right)$ such that $e^{(k+1)}_{\mathcal Q}(x,(\s,\tbf,\ubf))\in V_1\times\cdots\times V_{k+1}$. Recall that\[e_{\mathcal Q^{(k+1)},i}(x,(\s,\tbf,\ubf))=x+\sum_{I\in\mathcal Q^{(k)}\atop{i\in I}}s_I+\sum_{I\in\mathcal L^{(k)}\atop{i\in I}}t_I+\sum_{I\in\mathcal U^{(k)}\atop{i\in I}}u _I.\]
Therefore, we can rewrite the condition that $e^{(k+1)}_{\mathcal Q}(x,(\s,\tbf,\ubf))\in V_1\times\cdots\times V_{k+1}$ as
\[\left(y,\s\right)\in F^k(V_1,\ldots,V_k)\qquad\text{for}\qquad y=x+\sum_{I\in \mathcal U^{(k)}}u_I\]
and
\[x+\sum_{I\in \mathcal Q^{(k)}\atop{k+1\in I}}s_I+\sum_{I\in \mathcal L^{(k)}}t_I\in V_{k+1}.\]
Therefore, $f^{k+1}(V_1,\ldots,V_{k+1})$ is the number of edges between the multisets $W_k$ and $V_{k+1}$ in the multigraph $\Cay(\F_2^t,S)^m$ where $m=|\mathcal L^{(k)}\sqcup\mathcal U^{(k)}|$ and $W_k$ is the multiset with element \[y+\sum_{I\in\mathcal Q^{(k)}\atop{k+1\in I}}s_I\] for each $(y,\s)\in F^k(V_1,\ldots,V_k)$.

Write $w_x$ for the multiplicity of $x$ in $W_k$, noting that $w_x \leq |S|^{|\mathcal Q^{(k)}|}$ since this is the number of choices for $\s$. Hence, by the ordinary expander-mixing lemma (for multisets), we have
\begin{align*}
\left|f^{k+1}(V_1\vphantom{\frac{|S|^{|\mathcal Q^{(k)}|}}{2^t}}\right.
&\left.\!,\ldots,V_{k+1})-\frac{|S|^{|\mathcal Q^{(k+1)}|-|\mathcal Q^{(k)}|}}{2^t}f^k(V_1,\ldots,V_k)|V_{k+1}|\right|\\
&\leq (1-\epsilon)|S|^{|\mathcal Q^{(k+1)}|-|\mathcal Q^{(k)}|}\sqrt{\left(\sum_{x\in W_k}w_x^2\right)|V_{k+1}|}\\
&\leq (1-\epsilon)|S|^{|\mathcal Q^{(k+1)}|-|\mathcal Q^{(k)}|}\sqrt{|S|^{|\mathcal Q^{(k)}|}f^k(V_1,\ldots,V_k)|V_{k+1}|}\\
&\leq (1-\epsilon)|S|^{|\mathcal Q^{(k+1)}|}\sqrt{\lvert V_1\rvert\lvert V_{k+1}\rvert}.
\end{align*}
The last line follows from the easy bound $f^k(V_1,\ldots, V_k)\leq |S|^{|\mathcal Q^{(k)}|}|V_1|$. To see this inequality, note that $|S|^{|\mathcal Q^{(k)}|}$ is the number of choices for $\s$. Once $\s$ is chosen, for each $x\in V_1$, there is a unique $y\in\F_2^t$ such that $e_{\mathcal Q^{(k)},1}(y,\s)=x$.

Since $f^1(V_1)=|V_1|$, we can telescope this bound to conclude that
\begin{align*}
\left|f^r(V_1\vphantom{\frac{|S|^{|\mathcal Q^{(r)}|}}{(2^t)^{r-1}}}\right.
&\left.\!,\ldots,V_r)-\frac{|S|^{|\mathcal Q^{(r)}|}}{(2^t)^{r-1}}|V_1|\cdots|V_{r}|\right|\\
&\leq(1-\epsilon)|S|^{|\mathcal Q^{(r)}|}\sum_{i=2}^{r}\sqrt{\lvert V_1\rvert\lvert V_i\rvert}\prod_{j=i+1}^r\frac{|V_j|}{2^t}\\
&\leq(1-\epsilon)(r-1)|S|^{|\mathcal Q^{(r)}|}\sqrt{\lvert V_1\rvert\lvert V_r\rvert}.
\end{align*}
To complete the proof, note that $f_{\mathcal Q}(V_1,\ldots,V_r)=|S|^{|\mathcal Q|-|\mathcal Q^{(r)}|}f^r(V_1,\ldots,V_r)$, since $Q\setminus Q^{(r)}$ is a multiset supported on $\{\emptyset,[r]\}$ and these terms do not contribute materially to $e_{\mathcal Q}(x,\s)$.
\end{proof}

\begin{prop}
\label{thm:discrepancy}
For $r\geq 3$, $t$, $\epsilon>0$ and $S\subset \F_2^t$ satisfying Assumptions \ref{asmp1}, \ref{asmp2}, \ref{asmp3} and $V_1,\ldots,V_r\subseteq \F_2^t$,
\[\begin{split}\left|e_{\Hsf_{r,t,S}}(V_1,\ldots,V_r)\vphantom{\frac{|S|}{(2^t)^{r-1}}}-\right.
&\left.\frac{|S|(|S|-1)\cdots(|S|-(2^{r-1}-2))}{(2^t)^{r-1}}|V_1|\cdots|V_r|\right|\\
&\leq(1-\epsilon)2r|S|^{2^{r-1}-1}\sqrt{\lvert V_1\rvert\lvert V_r\rvert}.
\end{split}\]
\end{prop}

\begin{proof}
For $\Lambda$ a partition of $\mathcal P$, let $\mathcal T[\Lambda]\subseteq S^{\mathcal P}$ be the set of tuples $(s_I)_{I\in\mathcal P}$ such that if $I,J$ are in the same part of $\Lambda$ then $s_I=s_J$. We define a multiset $\mathcal P[\Lambda]$ supported in $2^{[r]}$ of size $|\Lambda|$ by letting the symmetric difference $I_1\triangle\cdots\triangle I_k$ (that is, the set of elements that appear in an odd number of the $I_i$'s) be an element of $\mathcal P[\Lambda]$ for each $\{I_1,\ldots,I_k\}$ a part of $\Lambda$.

For $\s\in T[\Lambda]$, define $\tbf\in S^{\mathcal P[\Lambda]}$ by letting $t_{I_1\triangle\cdots\triangle I_k}$ be equal to $s_{I_1}=\cdots=s_{I_k}$ whenever $\{I_1,\ldots,I_k\}$ is a part of $\Lambda$. This gives a bijection between $T[\Lambda]$ and $S^{\mathcal P[\Lambda]}$. Moreover, this map has the property that $e(x,\s)=e_{\mathcal P[\Lambda]}(x,\tbf)$ whenever $\s\in T[\Lambda]$. Thus, the number of pairs $(x,\s)\in\F_2^t\times \mathcal T[\Lambda]$ such that $e(x,\s)\in V_1\times\cdots\times V_r$ is exactly equal to $f_{\mathcal P[\Lambda]}(V_1,\ldots,V_r)$.

Finally, we want to count $e(V_1,\ldots,V_r)$, the number of $(x,\s)\in\F_2^t\times S^{|\mathcal P|}$ with $e(x,\s)\in V_1\times\cdots\times V_r$ and the additional constraint that the $s_I$'s are distinct. By M\"obius inversion, it follows that
\[e(V_1,\ldots,V_r)=\sum_{\Lambda}\mu(\Lambda)f_{\mathcal P[\Lambda]}(V_1,\ldots,V_r),\]
where 
\[\mu(\Lambda)=(-1)^{|\mathcal P|-|\Lambda|}\prod_{L\in\Lambda}(|L|-1)!.\] 
Combining this formula with the result of Lemma \ref{thm:disc-with-rep} gives that
\begin{equation*}
\begin{split}
\left|e(V_1,\ldots,V_r)-\frac{|V_1|\cdots|V_r|}{(2^t)^{r-1}}\right.
&\left.\vphantom{\frac{|V_1|}{(2^t)^{r-1}}}\sum_{\Lambda}\mu(\Lambda)|S|^{|\Lambda|}\right|\\
&\leq(1-\epsilon)r\sqrt{\lvert V_1\rvert\lvert V_r\rvert}\sum_{\Lambda}|\mu(\Lambda)||S|^{|\Lambda|}.
\end{split}
\end{equation*}
The first sum is exactly equal to $|S|(|S|-1)\cdots(|S|-(2^{r-1}-2))$, while we can bound the second sum by $2|S|^{2^{r-1}-1}$ in the same manner as in the proof of Lemma \ref{thm:asmp3-lower}.
\end{proof}

Averaging the above result over all renumberings of $V_1,\ldots,V_r$, we obtain Theorem \ref{thm:avg-discrepancy}.

\section{Concluding remarks}

We conclude with two questions. The first is essentially a reiteration of a question in~\cite{C17}: are the hypergraphs $\Hsf_{r, t, S}$ topological expanders for some suitable choice of parameters? In~\cite{C17}, the first author speculated that they might even satisfy a certain combinatorial notion of expansion, known as cosystolic expansion, from which topological expansion follows~\cite{DKW16}. Unfortunately, as pointed out to us by Gundert and Luria~\cite{GL18}, this is not the case. Nevertheless, the possibility that our hypergraphs are topological expanders remains a tantalizing one.

The second question concerns generalizations of our construction. In a fairly precise sense, abelian groups are the worst possible groups over which to define Cayley graphs if one is trying to produce expanders. A much better choice would be to work over finite simple groups of Lie type with bounded rank, where it is known~\cite{BGGT15} that two random elements almost surely suffice to generate an expander. However, our constructions depend in an absolutely critical way on commutativity, so much so that they might fittingly be called abelian complexes. On the other hand, $r$-uniform Ramanujan complexes can be explicitly rendered~\cite{LSV052} in terms of Cayley graphs over ${\rm PGL}_r(\F_q)$, so it is reasonable to believe that there is some more general mechanism which works over these groups. We consider the problem of distilling out this mechanism to be one of the central problems in the area.

\vspace{4mm}
\noindent
{\bf Acknowledgements.} This paper was partially written while the first author was visiting the California Institute of Technology as a Moore Distinguished Scholar and he is extremely grateful for their kind support.

\end{document}